\documentclass{amsart}
\usepackage{amsfonts,amssymb}
\usepackage{enumerate,graphicx}
\usepackage{caption}
\usepackage{subcaption}
\usepackage[square,numbers]{natbib}
\usepackage{mathrsfs,bbm}
\usepackage{stmaryrd}
\usepackage{tikz,tikzscale,tikz-cd}
\usepackage{multirow,xparse,fp}
\usepackage{environ}
\usepackage{amstext}
\usepackage{hyperref}
\usepackage{diagbox}
\usepackage[pagewise]{lineno}
\usepackage{comment}
\newtheorem{theorem}{Theorem}[section]

\newtheorem{corollary}[theorem]{Corollary}
\newtheorem{lemma}[theorem]{Lemma}
\theoremstyle{definition}

\newtheorem{definition}[theorem]{Definition}
\newtheorem{example}[theorem]{Example}

\newtheorem{remark}[theorem]{Remark}

\numberwithin{equation}{section}
\allowdisplaybreaks[4]

\begin{document}
	\title{Entropy of axial product of multiplicative subshifts}
	
	\author[Jung-Chao Ban]{Jung-Chao Ban}
	\address[Jung-Chao Ban]{Department of Mathematical Sciences, National Chengchi University, Taipei 11605, Taiwan, ROC.}
	\address{Math. Division, National Center for Theoretical Science, National Taiwan University, Taipei 10617, Taiwan. ROC.}
	\email{jcban@nccu.edu.tw}
	
	\author[Wen-Guei Hu]{Wen-Guei Hu}
	\address[Wen-Guei Hu]{College of Mathematics, Sichuan University, Chengdu, 610064, P. R. China}
	\email{wghu@scu.edu.cn}
	\author[Guan-Yu Lai]{Guan-Yu Lai}
	\address[Guan-Yu Lai]{Department of Mathematical Sciences, National Chengchi University, Taipei 11605, Taiwan, ROC.}
	\email{gylai@nccu.edu.tw}

 \author[Lingmin Liao]{Lingmin Liao}
	\address[Lingmin Liao]{School of Mathematics and Statistics, Wuhan University, Wuhan, Hubei 430072, P. R. China}
	\email{lmliao@whu.edu.cn}

	\keywords{topological entropy, surface entropy, axial product, multiplicative subshifts}
	
	\thanks{Ban is partially supported by the National Science and Technology Council, ROC (Contract No NSTC 111-2115-M-004-005-MY3 and 109-2115-M-390-003-MY3) and National Center for Theoretical Sciences. Hu is partially supported by the National Natural Science Foundation of China (Grant No.12271381). Lai is partially supported by the National Science and Technology Council, ROC (Contract NSTC 111-2811-M-004-002-MY2).}
	
	
	\baselineskip=1.2\baselineskip
	
	\begin{abstract}
		We obtain the entropy and the surface entropy of the axial products on $\mathbb{N}^d$ and the $d$-tree $T^d$ of two types of systems: the subshift and the multiplicative subshift. 
	\end{abstract}
	\maketitle
	
\section{Introduction}

\subsection{Motivations}

Let $\mathcal{A}$ be a finite set and $X_{1},\ldots
,X_{d}\subseteq \mathcal{A}^{\mathbb{N}}$ be $d$ subshifts. The \emph{axial
product $\otimes _{i=1}^{d}X_{i}=X_{1}\otimes \cdots \otimes X_{d}\subseteq \mathcal{A}^{\mathbb{N}^{d}}$ of subshifts }$X_{1},\ldots ,X_{d}$\emph{\ on }$\mathbb{N}^{d}$ is defined by
\[
\otimes _{i=1}^{d}X_{i}=\left\{x\in \mathcal{A}^{\mathbb{N}^{d}}:\forall g\in 
\mathbb{N}^{d}\text{, }\forall i\in \{1,\ldots ,d\}\text{, }(x_{g+je_i})_{j=0}^\infty\in X_{i}\right\}\text{,}
\]
where $\{e_{1},\ldots ,e_{d}\}$ is the standard basis of $\mathbb{N}^{d}$. Suppose $T^{d}$ is a conventional $d$-tree%
\footnote{By a \emph{tree} we mean an infinite, locally finite, connected graph with a distinguished vertex, say root $\epsilon $ and without loop or cycles. Here, we only consider trees without leaves.}. That is, $T^{d}$ is a free semigroup generated by $\Sigma =\{f_{1},\ldots ,f_{d}\}$ with the unit, say $\epsilon$. The \emph{axial product $\times _{i=1}^{d}X_{i}=X_{1}\times \cdots \times X_{d}\subseteq \mathcal{A}^{T^{d}}$ of subshifts }$X_{1},\ldots
,X_{d}\subseteq \mathcal{A}^{\mathbb{N}}$\emph{\ on }$T^{d}$ is defined by 
\[
\times _{i=1}^{d}X_{i}=\left\{x\in \mathcal{A}^{T^d}:\forall g\in 
T^d\text{, }\forall i\in \{1,\ldots ,d\}\text{, }(x_{gf_{i}^{j}})_{j=0}^\infty\in X_{i}\right\}.
\]

An axial product $\otimes _{i=1}^{d}X_{i}$ (or $\times _{i=1}^{d}X_{i}$) is
called \emph{isotropic} if $X_{i}=X_{j}$ $\forall 1\leq i\neq j\leq d$, and
is called \emph{anisotropic} if it is not isotropic\footnote{An isotropic axial product space is also called \emph{hom shifts} on $\mathbb{N}^{d}$ \cite{chandgotia2016mixing} or on $T^d$ \cite{petersen2020entropy}.}. The isotropic axial product of shifts on $\mathbb{N}^{d}$ was introduced in many important physical systems. For example, the hard square model on $\mathbb{N}^{2}$ (or on $\mathbb{Z}^{2}$)  are this kind of multidimensional shifts \cite{baxter2000solvable,baxter2016exactly}.

Let $\Lambda \subseteq \mathbb{N}^{d}$ ($d\geq 1$) be a finite set and 
$X\subseteq \mathcal{A}^{\mathbb{N}^{d}}$ be a subshift. We denote by $\mathcal{P}(\Lambda ,X)$ the set of \emph{canonical projections} from any $x\in X$ onto $\mathcal{A}^{\Lambda }$. That is, 
\begin{equation}\label{P}
    \mathcal{P}(\Lambda ,X)=\left\{(x_{g})_{g\in \Lambda}\in \mathcal{A}^{\Lambda }:x\in X\right\}.
\end{equation}
Denote $[1,n]:=\{1,...,n\}$, and $\left\llbracket(1,\ldots ,1),(n_1,\ldots ,n_d)\right\rrbracket:=[1,n_1]\times\cdots\times [1,n_d]$. The \emph{entropy} of $X$ is defined as 
\begin{equation}
h(X)=\lim_{n\rightarrow \infty }\frac{\log \left\vert \mathcal{P}(\left\llbracket(1,\ldots,1),(n,\ldots ,n)\right\rrbracket,X)\right\vert }{\left\vert \left\llbracket(1,\ldots ,1),(n,\ldots,n)\right\rrbracket\right\vert }\text{,}  \label{4}
\end{equation}
where $\left\vert \cdot
\right\vert $ stands for the cardinality of a set. The limit (\ref{4}) exists since $\mathbb{N}^{d}$ is an amenable group. We refer the reader to \cite{ceccherini2010cellular} for more details about the shifts defined on amenable groups.

Let $X\subseteq \mathcal{A}^{T^{d}}$ be an axial product on $T^{d}$. By
abuse of notation, we write $\left\vert g\right\vert $ the number of edges
from $\epsilon $ to $g\in T^{d}$. Define $\Delta _{n}=\{g\in
T^{d}:\left\vert g\right\vert \leq n\}$ and $T_{n}=\{g\in T^{d}:\left\vert
g\right\vert =n\}$. The \emph{entropy }of $X\subseteq 
\mathcal{A}^{T^{d}}$ is defined by 
\begin{equation}\label{5}
h^{T}(X)=\lim_{n\rightarrow \infty }\frac{\log \left\vert \mathcal{P}(\Delta
_{n},X)\right\vert }{\left\vert \Delta _{n}\right\vert }\text{.}  
\end{equation}
Note that the limit (\ref{5}) defining entropy exists for tree-shifts defined on $T^d$ \cite{PS-2017complexity, petersen2020entropy}. Such entropy also exists for shifts on another tree called Markov--Cayley tree (see \cite{ban2022stem} for the definition and other details). 

We emphasize that the values of $h(X)$ and $h^T(X)$ have significant meanings in statistical physics  \cite{baxter2000solvable,
baxter2016exactly, dembo2010ising, dembo2013factor, georgii2011gibbs,
higuchi1977remarks, lyons1989ising, spitzer1975markov}. However, it is challenging to compute the explicit values of $h(X)$ and $h^{T}(X)$ \cite{ban2021structure,MP-ETaDS2013}. It is worth noting that since the tree $T^{d}$ lacks the amenability, i.e., $\left\vert \Delta_{n}\backslash \Delta _{n-1}\right\vert/\left\vert \Delta _{n}\right\vert $ does not tend to $0$ as $n\rightarrow\infty$, the value $h^{T}(X)$ behaves quite differently from $h(X)$ (cf. \cite{ban2021structure}).

On the other hand, the limiting entropy and independence entropy of axial product $\otimes _{i=1}^{d}X_{i}$ as $d\rightarrow \infty$ have been introduced and studied in \cite{louidor2013independence, meyerovitch2014independence}.

The topological dynamics and the mixing properties of $\otimes_{i=1}^{d}X_{i}$ and $\times_{i=1}^{d}X_{i}$ also attract a lot of attention. In \cite{chandgotia2016mixing}, the authors studied the decidability for some topological properties of $\otimes_{i=1}^{2}X_{i}$. The relationship between the mixing properties of $\times _{i=1}^{2}X$ and those of $X$ has been discussed in \cite{ban2021characterization}.

The main aim of the present paper is to establish the entropy formulae of $h(X\otimes Y)$ on $\mathbb{N}^{2}$ and $h^{T}(X\times Y)$ on $T^{2}$ respectively. The candidate shifts $X$ and $Y$ are either a subshift of finite type (SFT) or a multiplicative shift. By the same methods, we also give the entropy formulae of axial products for $d\geq 3$.

Driven by the multifractal analysis of the `nonconventional ergodic average' \cite{kifer2010nonconventional, kifer2014nonconventional} or the `multiple ergodic theory' \cite{furstenberg2014recurrence, furstenberg1982ergodic,furstenberg1978topological, host2018nilpotent}, Fan, Liao and Ma \cite{fan2012level} initiated the investigation of the \emph{multiplicative shift} 
\begin{equation}
X^{2}=\{x\in \{0,1\}^{\mathbb{N}}:x_{i}x_{2i}=0\text{ for all }i\in \mathbb{N}\}\text{,}  \label{6}
\end{equation}%
and obtained the entropy formula  
\[
h(X^{2})=\frac{1}{2}\sum_{n=1}^{\infty }\frac{\log F_{n}}{2^{n}}\text{,}
\]
where $\{F_{n}\}_{n=0}^{\infty }$ is the Fibonacci sequence defined by $F_{0}=F_{1}=1$ and $F_{n+1}=F_{n}+F_{n-1}$ for all $n\geq 1$. Afterwards, Kenyon, Peres and Solomyak \cite{kenyon2012hausdorff} setted the general form of (\ref{6}) as 
\begin{equation}
X_{\Omega }^{p}=\{x\in \{0,1\}^{\mathbb{N}}:(x_{ip^{l}})_{l=0}^{\infty }\in\Omega \text{ for all } i\in\mathbb{N},\text{ }p\nmid i\}\text{,}  \label{7}
\end{equation}
where $\Omega \subseteq \{0,1\}^{\mathbb{N}}$ is a subshift or a SFT. It can be easily checked that $X^{2}=X_{\Omega }^{2}$, where $\Omega =\Sigma_{A}$ is the SFT with $A=\left[ 
\begin{array}{cc}
1 & 1 \\ 
1 & 0%
\end{array}
\right]$. In \cite{kenyon2012hausdorff}, the set (\ref{7}) is called a \emph{multiplicative subshift} (resp. \emph{multiplicative SFT}) if $\Omega $ is a subshift (resp. SFT) of $\{0,1\}^{\mathbb{N}}$. Multiplicative shifts have been an active research topic since then, and have attracted much attention (cf. \cite{ ban2021entropy,ban2019pattern, ban2023entropy, brunet2021dimensions,fan2016multifractal, peres2012dimension}).

As for the entropy formula of the axial product of two multiplicative
shifts. Ban, Hu and Lai \cite{ban2021entropy} obtained the entropy formula for $X_{\Omega _{1}}^{p_{1}}\otimes X_{\Omega _{2}}^{p_{2}}$ on $\mathbb{N}^{2}$, where $2\leq p_{1},p_{2}\in \mathbb{N}$, and $\Omega_{1},\Omega _{2}\subseteq \{0,1\}^{\mathbb{N}}$ (Theorem \ref{thm2.1}). It is worth noting that the main difference between $X_{\Omega }^{p}$ and the subshift $\Omega $ is that the subshift $\Omega $ is invariant under `the action of the additive semigroup of positive integers', while the set $X_{\Omega }^{p}$ is invariant under `the action of multiplicative semigroup of positive integers' (i.e., $(x_{k})_{k=1}^{\infty }\in X_{\Omega }$ implies $(x_{rk})_{k=1}^{\infty }\in X_{\Omega}$
for all $r\in \mathbb{N}$, cf. \cite{kenyon2012hausdorff}).

\subsection{Results}
The present paper aims at answering the following two questions.
\begin{enumerate}
\item[(Q1)] How the underlying lattice ($\mathbb{N}^{2}$ or $T^{2}$) affects the entropy?

\item[(Q2)] How do the two actions (additive and multiplicative semigroups)
affect the entropy when the underlying lattice is fixed?
\end{enumerate}

 We will establish the entropy formula for the axial product $X_{\Omega _{1}}^{p_{1}}\times X_{\Omega_{2}}^{p_{2}}$ on $T^2$ (Theorem \ref{thm 3-1}). Then comparing the entropy formula for $X_{\Omega_{1}}^{p_{1}}\otimes X_{\Omega _{2}}^{p_{2}}$ on $\mathbb{N}^{2}$ (Theorem \ref{thm2.1}) answers (Q1). As for (Q2), Theorem \ref{thm 2-2} (resp. Theorem \ref{thm 3-2}) gives the entropy formula for $X_{\Omega}^{p}\otimes X$ (resp. $X_{\Omega}^{p}\times X$) on $\mathbb{N}^{2}$ (resp. $T^{2}$). 
 
 The challenges and importance of the investigation are described
below. It is known that the main steps in finding the entropy formulae of the multiplicative shifts on $\mathbb{N}$ (or $\mathbb{N}^{2}$) are\\
\textbf{(a)} decomposing the whole lattice according to the multiplicative constraint, e.g., 
$x_{i}x_{2i}=0$, and\\
\textbf{(b)} calculating the densities of the independent
sub-lattices from \textbf{(a)} among the whole lattice.\\ 
However, unlike the discussion of the multiplicative shifts defined on $\mathbb{N}$ (or $\mathbb{N}^{2}$), the decompositions of $X_{\Omega_{1}}^{p_{1}}\times X_{\Omega
_{2}}^{p_{2}}$ and $X_{\Omega }^{p}\times X$ on $T^{2}$ are quite different. For $X_{\Omega_{1}}^{p_{1}}\times X_{\Omega
_{2}}^{p_{2}}$ or $X_{\Omega }^{p}\times X$, we need to decompose $T^2$ into infinitely many disjoint subtrees. On the other hand, the computation of densities of these disjoint sub-lattices among the whole lattice from the step \textbf{(a)} reduces to an integer partition problems \cite{andrews1976}. The study of partitions of an integer was started by Euler in the mid-eighteenth century \cite{Euler1748}. Euler introduced the idea of a generating function of a sequence $\{a_n\}_{n=0}^\infty$, that is, he encoded the sequence as the coefficients of a power series $\sum_{n=0}^\infty a_n x^n$. In particular, Euler showed that the generating function for $p(n)$, the number of partitions of $n$, can be expressed as an elegant infinite product:
\begin{equation*}
    \sum_{n=0}^\infty p(n)x^n=\prod_{j=1}^\infty \frac{1}{1-x^j}.
\end{equation*}
In 1918, Hardy and Ramanujan \cite{hardy1918aymp} also obtained the asymptotic formula of the number of compositions (ordered partitions) of $n$ into $\mathbb{N}$. In this paper, we will introduce and study some general types of integer partition. The asymptotic behavior of types of partitions will become challenges in the field of combinatorics (see Section \ref{sec3}).

\subsection{Applications: surface entropy}
One important consequence of Theorems \ref{thm2.1}, \ref{thm 2-2}, \ref{thm 3-1} and \ref{thm 3-2} leads us to study the surface entropy of the axial products $X_{\Omega _{1}}^{p_{1}}\otimes X_{\Omega
_{2}}^{p_{2}}$, $X_{\Omega }^{p}\otimes X$, $X_{\Omega _{1}}^{p_{1}}\times X_{\Omega
_{2}}^{p_{2}}$ and $X_{\Omega }^{p}\times X$. The concept of surface
entropy was first introduced by Pace \cite{pace2018surface}. Let $X$ be a shift on $\mathbb{N}^{2}$, the \emph{surface entropy} of $X$ with \emph{eccentricity} $\alpha $ is 
\[
h_{s}(X,\alpha )=\sup_{\{(x_{n},y_{n})\}\in \Gamma _{\alpha
}}\limsup\limits_{n\rightarrow \infty }S_{X}(x_{n},y_{n})\text{,}
\]%
where 
\[
S_{X}(x_{n},y_{n})=\frac{\log \left\vert \mathcal{P}%
(\left\llbracket(1,1),(x_{n},y_{n})\right\rrbracket,X)\right\vert -x_{n}y_{n}h(X)}{x_{n}+y_{n}}\text{,}
\]%
and $\Gamma _{\alpha }=\{\{(x_{n},y_{n})\}\in \left( \mathbb{N}^{2}\right)^{\mathbb{N}}:\frac{y_{n}}{x_{n}}\rightarrow \alpha $ and $x_{n}\rightarrow
\infty \}$. The main purpose of the study of surface entropy is to understand the `linear term' of the complexity function $(x_n,y_n)\mapsto\log \left\vert \mathcal{P}(\left\llbracket(1,1),(x_{n},y_{n})\right\rrbracket,X)\right\vert $. The rigorous surface entropy formula for SFTs on $\mathbb{N}$ and a lot of interesting approximation properties of surface entropy of SFTs on $\mathbb{N}^{2}$ were derived in \cite{pace2018surface}. Some related results for the surface entropy may also be found in \cite{callard2021computational,meyerovitch2011growth}. 

Using the entropy formula obtained in Theorems \ref{thm2.1}, \ref{thm 2-2}, \ref{thm 3-1} and \ref{thm 3-2}, we derive the convergence rates of the following quantities (see Theorem \ref{thm4.1}): 
\begin{eqnarray}
&&\log \left\vert \mathcal{P}(\left\llbracket(1,1),(m,n)\right\rrbracket,X_{\Omega _{1}}^{p_{1}}\otimes
X_{\Omega _{2}}^{p_{2}})\right\vert -mnh(X_{\Omega _{1}}^{p_{1}}\otimes
X_{\Omega _{2}}^{p_{2}}),  \label{8} \\
&&\log \left\vert \mathcal{P}(\left\llbracket(1,1),(m,n)\right\rrbracket,X_{\Omega }^{p}\otimes
X)\right\vert -mnh(X_{\Omega }^{p}\otimes X),  \label{9} \\
&&\log \left\vert \mathcal{P}(\Delta _{n},X_{\Omega _{1}}^{p_{1}}\times
X_{\Omega _{1}}^{p_{1}})\right\vert -\left\vert \Delta _{n}\right\vert
h^{T}(X_{\Omega _{1}}^{p_{1}}\times X_{\Omega _{1}}^{p_{1}}),  \label{10} \\
&&\log \left\vert \mathcal{P}(\Delta _{n},X_{\Omega }^{p}\times
X)\right\vert -\left\vert \Delta _{n}\right\vert h^{T}(X_{\Omega
_{1}}^{p_{1}}\times X_{\Omega _{1}}^{p_{1}}).  \label{11}
\end{eqnarray}
In Theorem \ref{thm4.1}, we prove that the convergence rates of (\ref{8}) and (\ref{9}) are of $O(n\log m+m\log n)$ and the convergence rates of (\ref{10}) and (\ref{11}) are of $O(B^{n})$ where $B=\frac{1}{|x_0|}= 1.80193...$ and $x_0$ is the root of the polynomial $x^3-2x^2-x+1$ with the smallest absolute value. The difference between $O(n\log m+m\log n)$ and $O(B^{n})$ comes from the structures of $\mathbb{N}^{2}$ and $T^{2}$. We remark that the value $B$ is not sharp. The sharp value is closely related to finding an asymptotic formula of some integer partition problems, for example, binary partitions \cite{oystein2002} (see Lemmas \ref{lem 3-1} and \ref{lem 3-4}).

\section{Entropy of axial products on $\mathbb{N}^2$}
In this section, we consider the entropy of axial products on $\mathbb{N}^2$. We first present a known result of Ban et al. \cite[Theorem 1.6]{ban2021entropy} on the entropy of axial product of multiplicative subshifts on $\mathbb{N}^2$ (Theorem \ref{thm2.1}). Then, we give the entropy formula for axial products of subshift and multiplicative subshift on $\mathbb{N}^2$ in Theorem \ref{thm 2-2}. In order to present the results, the following notations are needed.  

Let $X,\Omega,\Omega_1,\Omega_2$ be subshifts and let $p,p_1,p_2$ be integers greater than 1.
\subsection{Entropy formula for $X_{\Omega_1}^{p_1}\otimes X_{\Omega_2}^{p_2}$} Recalling the definition of $\mathcal{P}(\cdot ,\cdot)$ in (\ref{P}), we have the following theorem.
\begin{theorem}[Ban et al., \cite{ban2021entropy}]\label{thm2.1}
The entropy of the axial product of $X_{\Omega_1}^{p_1}$ and $X_{\Omega_2}^{p_2}$ on $\mathbb{N}^2$ is 
\begin{equation*}
    h\left( X_{\Omega_1}^{p_1}\otimes X_{\Omega_2}^{p_2}\right)= \frac{(p_1-1)^2(p_2-1)^2}{(p_1p_2)^2}\sum_{i_1,i_2=1}^\infty \frac{\log \left|\mathcal{P}\left(\left\llbracket (1,1),\left(i_1,i_2\right)\right\rrbracket, \Omega_1\otimes \Omega_2\right)\right|}{p_1^{i_1-1}p_2^{i_2-1}}.
\end{equation*}
\end{theorem}
\begin{corollary}
     If $\Omega_1$ (resp. $\Omega_2$) is a full shift, then $h\left( X_{\Omega_1}^{p_1}\otimes X_{\Omega_2}^{p_2}\right)=h\left( X_{\Omega_2}^{p_2}\right)$ (resp. $h\left( X_{\Omega_1}^{p_1}\right)$).
\end{corollary}

\begin{proof}
    Since $\Omega_1$ is a full shift, for any $i_1,i_2\in\mathbb{N}$, we have
    \begin{equation}\label{2.1.1}
        \left|\mathcal{P}\left(\left\llbracket (1,1),\left(i_1,i_2\right)\right\rrbracket, \Omega_1\otimes \Omega_2\right)\right|=\left|\mathcal{P}\left( [1,i_2],\Omega_2\right)\right|^{i_1}.
    \end{equation}
    Combining (\ref{2.1.1}) and Theorem \ref{thm2.1}, we obtain
    \begin{align*}
        h\left( X_{\Omega_1}^{p_1}\otimes X_{\Omega_2}^{p_2}\right)=&\frac{(p_1-1)^2(p_2-1)^2}{(p_1p_2)^2}\sum_{i_1,i_2=1}^\infty \frac{\log \left|\mathcal{P}\left( [1,i_2],\Omega_2\right)\right|^{i_1}}{p_1^{i_1-1}p_2^{i_2-1}}\\
        =&\frac{(p_2-1)^2}{(p_2)^2}\sum_{i_2=1}^\infty \frac{\log \left|\mathcal{P}\left( [1,i_2],\Omega_2\right)\right|}{p_2^{i_2-1}}\\
        =&h\left( X_{\Omega_2}^{p_2}\right),
    \end{align*}
  where the second equality is due to the fact that for any $p> 1$,
  \begin{equation}\label{2.1.2}
      \frac{(p-1)^2}{p^2}\sum_{i=1}^\infty\frac{i}{p^{i-1}}=1.
  \end{equation}
  
   A similar argument can be applied to the case that $\Omega_2$ is a full shift. The proof is then completed.
\end{proof}

\subsection{Entropy formula for $X_{\Omega}^p \otimes X$} Concerning the entropy formula for $X_{\Omega}^p \otimes X $, we have the following theorem.
\begin{theorem}\label{thm 2-2}
The entropy of the axial product of $X_{\Omega}^{p}$ and $X$ on $\mathbb{N}^2$ is given by
\begin{equation*}
    h\left(X_{\Omega}^p \otimes X \right)=\frac{(p-1)^2}{p^2}\sum_{i=1}^\infty \frac{\lambda_i}{p^{i-1}},
\end{equation*}
where 
\begin{equation*}
\lambda_{i}=\lim_{j\to\infty}\frac{\log\left|\mathcal{P}\left(\left\llbracket (1,1),\left(i,j\right)\right\rrbracket, \Omega\otimes X\right)\right|}{j},
\end{equation*}
and $\mathcal{P}(\cdot, \cdot)$ is defined as in (\ref{P}).
\end{theorem}
\begin{remark}
    Note that for any $i\geq 1$ and any $j_1,j_2\geq 1$
    \begin{align*}
&\left|\mathcal{P}\left(\left\llbracket (1,1),\left(i,j_1+j_2\right)\right\rrbracket, \Omega\otimes X\right)\right|\\
\leq& \left|\mathcal{P}\left(\left\llbracket (1,1),\left(i,j_1\right)\right\rrbracket, \Omega\otimes X\right)\right|\left|\mathcal{P}\left(\left\llbracket (1,1),\left(i,j_2\right)\right\rrbracket, \Omega\otimes X\right)\right|.
\end{align*}
Hence, by \cite[Lemma 4.1.7]{LM-1995}, the limit $\lambda_i$ always exists for all $i\geq 1$.
\end{remark}

\begin{proof}
 Note that for any $m,n\geq 1$,
    \begin{equation}\label{disjoint1}
        \left\llbracket (1,1),\left(m,n\right)\right\rrbracket=\bigsqcup_{i\in\mathcal{I}_{p},i\leq m}\left\{i,...,ip^{\left\lfloor\log_{p} \frac{m}{i}\right\rfloor}\right\}\times [1,n],
    \end{equation}
where $\mathcal{I}_p=\left\{i\in \mathbb{N}: p\nmid i\right\}$. Since the rule defining $X_{\Omega}^p \otimes X $ on disjoint grids (\ref{disjoint1}) are independent, we have
    \begin{align*}
        &\left|\mathcal{P}\left(\left\llbracket (1,1),\left(m,n\right)\right\rrbracket,X_{\Omega}^p \otimes X \right)\right|\\
        =&\prod_{i\in\mathcal{I}_p,i\leq m}\left|\mathcal{P}\left(\left\llbracket (1,1),\left(\left\lfloor\log_{p} \frac{m}{i}\right\rfloor+1,n\right)\right\rrbracket,\Omega \otimes X \right)\right|\\
=&\prod_{k=1}^{\left\lfloor\log_{p}m\right\rfloor+1}\left|\mathcal{P}\left(\left\llbracket (1,1),\left( k,n\right)\right\rrbracket,\Omega \otimes X \right)\right|^{\left|\mathcal{I}_p\cap \left(\left\lfloor\frac{m}{p^k}\right\rfloor,\left\lfloor\frac{m}{p^{k-1}}\right\rfloor\right]\right|}\\
=&\prod_{k=1}^{\left\lfloor\log_{p}m\right\rfloor+1}\left|\mathcal{P}\left(\left\llbracket (1,1),\left( k,n\right)\right\rrbracket,\Omega \otimes X \right)\right|^{\frac{\left|\mathcal{I}_p\cap \left(\left\lfloor\frac{m}{p^k}\right\rfloor,\left\lfloor\frac{m}{p^{k-1}}\right\rfloor\right]\right|}{\left\lfloor\frac{m}{p^{k-1}}\right\rfloor-\left\lfloor\frac{m}{p^k}\right\rfloor}\left(\left\lfloor\frac{m}{p^{k-1}}\right\rfloor-\left\lfloor\frac{m}{p^k}\right\rfloor\right)}.
    \end{align*}
Thus, for $m,n\geq 1$,
    \begin{align*}
        &\frac{\log\left|\mathcal{P}\left(\left\llbracket (1,1),\left(m,n\right)\right\rrbracket,X_{\Omega}^p \otimes X \right)\right|}{\left|\left\llbracket (1,1),\left(m,n\right)\right\rrbracket\right|}=\sum_{k=1}^{\left\lfloor\log_{p}m\right\rfloor+1}f(m,n,k),
        \end{align*}
        where
        \begin{align*}
        f(m,n,k)=\frac{\frac{\left|\mathcal{I}_p\cap \left(\left\lfloor\frac{m}{p^k}\right\rfloor,\left\lfloor\frac{m}{p^{k-1}}\right\rfloor\right]\right|}{\left\lfloor\frac{m}{p^{k-1}}\right\rfloor-\left\lfloor\frac{m}{p^k}\right\rfloor}\left(\left\lfloor\frac{m}{p^{k-1}}\right\rfloor-\left\lfloor\frac{m}{p^k}\right\rfloor\right)\log\left|\mathcal{P}\left(\left\llbracket (1,1),\left(k,n\right)\right\rrbracket,\Omega \otimes X \right)\right|}{mn}.
        \end{align*}
        
        On the other hand, we have
        \begin{equation}\label{eq-12}
f(m,n,k)\leq\frac{\left(\left\lfloor\frac{m}{p^{k-1}}\right\rfloor-\left\lfloor\frac{m}{p^k}\right\rfloor\right)\log\left|\mathcal{A}\right|^{kn}}{mn}\leq\log\left|\mathcal{A}\right|\frac{k}{p^{k-1}}.
        \end{equation}
        Since the series of general term as the rightmost of (\ref{eq-12}) converges, by Weierstrass M-test, 
        \[\sum_{k=1}^\infty f(m,n,k)
        \]
        converges uniformly in $m$ and $n$.
        
        Therefore,
        \begin{align*}
             h\left(X_{\Omega}^p \otimes X \right)            =&\lim_{m,n\to\infty}\frac{\log\left|\mathcal{P}\left(\left\llbracket (1,1),\left(m,n\right)\right\rrbracket,X_{\Omega}^p \otimes X \right)\right|}{\left|\left\llbracket (1,1),\left(m,n\right)\right\rrbracket\right|}\\
        =&\lim_{m,n\to\infty}\sum_{k=1}^{\left\lfloor\log_{p}m\right\rfloor+1}f(m,n,k)\\
=&\sum_{k=1}^\infty\lim_{m,n\to\infty}f(m,n,k).
 \end{align*}
Applying \cite[Lemma 3.4]{ban2021entropy}, we have 
\[\lim_{m\to\infty}\frac{\left|\mathcal{I}_p\cap \left(\left\lfloor\frac{m}{p^k}\right\rfloor,\left\lfloor\frac{m}{p^{k-1}}\right\rfloor\right]\right|}{\left\lfloor\frac{m}{p^{k-1}}\right\rfloor-\left\lfloor\frac{m}{p^k}\right\rfloor}=\frac{p-1}{p}.
\]
Thus,
\begin{align*}
    &\lim_{m,n\to\infty}f(m,n,k)\\
    =&\frac{p-1}{p}\lim_{m,n\to\infty}\frac{\left(\left\lfloor\frac{m}{p^{k-1}}\right\rfloor-\left\lfloor\frac{m}{p^k}\right\rfloor\right)\log\left|\mathcal{P}\left(\left\llbracket (1,1),\left(k,n\right)\right\rrbracket,\Omega \otimes X \right)\right|}{mn}\\
         =&\frac{p-1}{p}\left(\frac{1}{p^{k-1}}-\frac{1}{p^k}\right)\lim_{m,n\to\infty}\frac{\log\left|\mathcal{P}\left(\left\llbracket (1,1),\left(k,n\right)\right\rrbracket,\Omega \otimes X \right)\right|}{n}\\
        =&\frac{(p-1)^2}{p^2} \frac{\lambda_k}{p^{k-1}}.
\end{align*}
Hence,
 \begin{align*}
 h\left(X_{\Omega}^p \otimes X \right)
        =\sum_{k=1}^\infty\lim_{m,n\to\infty}f(m,n,k)=\frac{(p-1)^2}{p^2}\sum_{k=1}^\infty \frac{\lambda_k}{p^{k-1}}.
        \end{align*}
\end{proof}
\begin{corollary}
     If $\Omega$ (resp. $X$) is a full shift, then $h\left(X_{\Omega}^p \otimes X \right)=h(X)$ (resp. $h\left(X_{\Omega}^p\right)$).
\end{corollary}
\begin{proof}
     Since $\Omega$ is a full shift, we have
    \begin{equation*}
        \left|\mathcal{P}\left(\left\llbracket (1,1),\left(i,j\right)\right\rrbracket, \Omega\otimes X\right)\right|=\left|\mathcal{P}\left([1,j], X\right)\right|^i.
    \end{equation*}
    Then,
    \begin{align*}
\lambda_{i}=&\lim_{j\to\infty}\frac{\log\left|\mathcal{P}\left(\left\llbracket (1,1),\left(i,j\right)\right\rrbracket, \Omega\otimes X\right)\right|}{j}\\
=&\lim_{j\to\infty}\frac{\log\left|\mathcal{P}\left([1,j], X\right)\right|^i}{j}\\
=&i\cdot h(X).
    \end{align*}
Thus, by Theorem \ref{thm 2-2} and (\ref{2.1.2}), we have    
\begin{align*}
     h\left(X_{\Omega}^p \otimes X \right)=\frac{(p-1)^2}{p^2}\sum_{i=1}^\infty \frac{\lambda_i}{p^{i-1}}=h(X)\frac{(p-1)^2}{p^2}\sum_{i=1}^\infty \frac{i}{p^{i-1}}=h(X).
\end{align*}

   Since $X$ is a full shift, we have
    \begin{equation*}
        \left|\mathcal{P}\left(\left\llbracket (1,1),\left(i,j\right)\right\rrbracket, \Omega\otimes X\right)\right|=\left|\mathcal{P}\left([1,i], \Omega\right)\right|^j.
    \end{equation*}
    Then,
    \begin{align*}
\lambda_{i}=&\lim_{j\to\infty}\frac{\log\left|\mathcal{P}\left(\left\llbracket (1,1),\left(i,j\right)\right\rrbracket, \Omega\otimes X\right)\right|}{j}\\
=&\lim_{j\to\infty}\frac{\log\left|\mathcal{P}\left([1,i], \Omega\right)\right|^j}{j}\\
=&\log\left|\mathcal{P}\left([1,i], \Omega\right)\right|.
    \end{align*}
Thus, by Theorem \ref{thm 2-2}, we have   
\begin{align*}
     h\left(X_{\Omega}^p \otimes X \right)=\frac{(p-1)^2}{p^2}\sum_{i=1}^\infty \frac{\log\left|\mathcal{P}\left([1,i], \Omega\right)\right|}{p^{i-1}}=h\left(X_{\Omega}^p \right).
\end{align*}
\end{proof}

\section{Entropy of axial products on the 2-tree}\label{sec3}
In this section, we consider the entropy of axial products on the 2-tree. More precisely, we calculate the entropy of axial product of a subshift and a multiplicative subshift on the 2-tree (Theorem \ref{thm 3-2}), and the entropy of axial product of two multiplicative subshifts on the 2-tree (Theorem \ref{thm 3-1}). In order to present the results, the following definitions are needed. 

Let $T^2$ be the 2-tree, i.e., $T^2=\langle f_1,f_2\rangle$ is the semigroup generated by two generators $f_1$ and $f_2$ with unit $\epsilon$. Let $T_0=\{\epsilon\}$ and for $n\geq 1$, let
\begin{equation*}
    T_n=\left\{g\in T^2: g=f_{i_1}\cdots f_{i_n}, i_j\in \{1,2\},\forall 1\leq j\leq n \right\}.
\end{equation*}
 For $n\geq 0$, let $\Delta_n=\cup_{i=0}^n T_n$. In order to decompose $\Delta_n$ by the multiplicative rule, we define a type for each vertex in $T^2$. For each vertex $g$ in $T^2$, we say that $g$ has type $(i,1)$ if $g=f_1^{i-1}$ or $g=\cdots f_2f_1^{i-1}$, and we say that $g$ has type $(1,j)$ if $g=f_2^{j-1}$ or $g=\cdots f_1f_2^{j-1}$. For example, $\epsilon$ has type $(1,1)$, $g=f_1$ has type $(2,1)$ and $g=f_2$ has type $(1,2)$ (see Figure \ref{fig 1}\begin{figure}
		\includegraphics[scale=0.55]{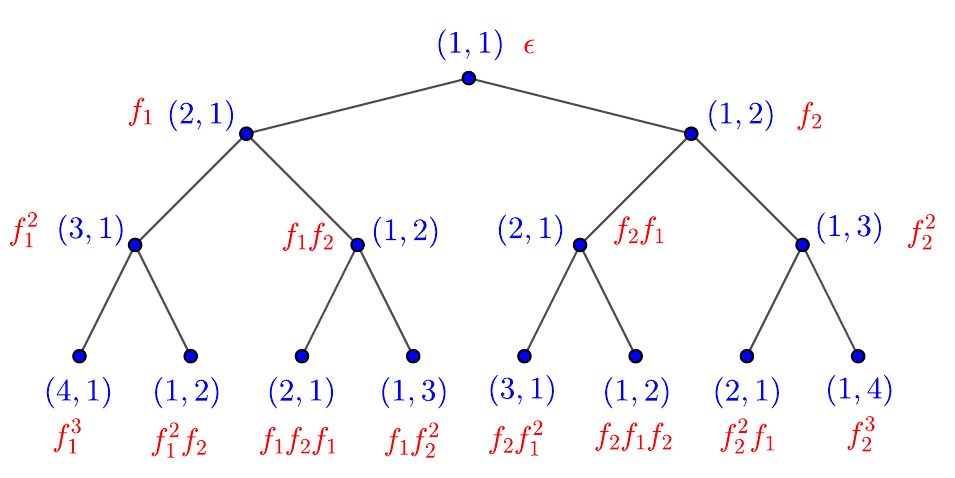}
		\caption{The types of $\Delta_3$ of $T^2$.}\label{fig 1}
	\end{figure} for the types of $\Delta_3$ of $T^2$). More intuitively, ``$g=f_1^{i-1}$ has type $(i,1)$'' means that $g$ is the $i$-th vertex of the path $\{\epsilon,f_1,f_1^2,...\}$, thus, the first component of $(i,1)$ is $i$; and $g$ is the initial vertex of the path $\{g,gf_2,gf_2^2,...\}$, thus, the second component of $(i,1)$ is 1. That ``$g=g' f_2f_1^{i-1}$ has type $(i,1)$'' means that $g$ is the $i$-th vertex of the path $\{g'f_2,g'f_2f_1,g'f_2f_1^2,...\}$, thus, the first component of $(i,1)$ is $i$; and $g$ is the first vertex of the path $\{g,gf_2,gf_2^2,...\}$, thus, the second component of $(i,1)$ is 1. 
 
 \subsection{Entropy formula for $X_{\Omega_1}^{p_1}\times X_{\Omega_2}^{p_2}$}
Note that the multiplicative subshifts are not shift invariant, the definition of axial product $X_{\Omega_1}^{p_1}\times X_{\Omega_2}^{p_2}$ is different from the definition of the axial product of subshifts, and is defined by
\[
\left\{x\in \mathcal{A}^{T^2}:\forall g\in 
T^2\text{, }\forall i\in \{1,2\}\mbox{ with }f_i\neq e(g) \text{, }(x_{gf_{i}^{j}})_{j=0}^\infty\in X_{\Omega_i}^{p_i}\right\},
\]
where $e(g)$ is the end generator of the vertex $g$. 
 
 For $k\geq 0$ and $i\in\mathbb{N}$, define $\Delta_k^{(i,1)}$ with respect to the multiplicative constraints $p_1$ and $p_2$ to be the reshaped set of the subset of $g\Delta_k:=\{gg':g'\in \Delta_k\}$ such that $g$ is a vertex of type $(i,1)$ and each vertex in this subset is dependent on $g$ with respect to the multiplicative constraints $p_1$ and $p_2$ (see Figure \ref{fig 2}  \begin{figure}
		\includegraphics[scale=0.55]{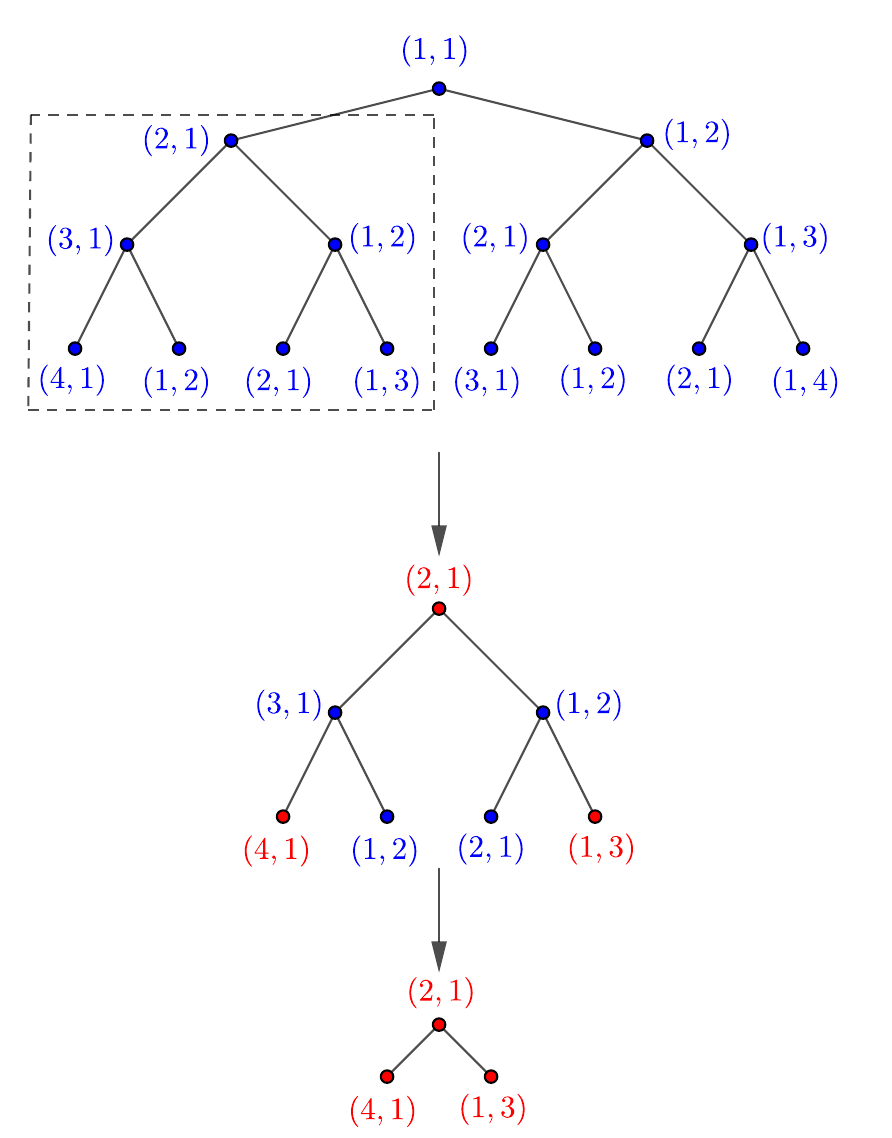}
		\caption{The $\Delta_2^{(2,1)}$ with respect to the multiplicative constraints $p_1=2$ and $p_2=3$.}\label{fig 2}
	\end{figure}for $\Delta_2^{(2,1)}$ with respect to the multiplicative constraints $p_1=2$ and $p_2=3$). Similarly, we define $\Delta_k^{(1,j)}$ with respect to the multiplicative constraints $p_1$ and $p_2$ for $k\geq 0$ and $j\in\mathbb{N}$. In order to estimate the cardinality of $\Delta_k^{(i,1)}$ (or $\Delta_k^{(1,j)}$), the following definition is needed. 
\begin{definition}
    For $n\geq 1$ and $p_1,p_2\geq 2$, we define $P_{p_1,p_2}(n)$ to be the set of ordered partitions of $n$ into $P_1:=\{p_1^k-1:k\geq 1\}$ and $P_2:=\{p_2^k-1:k\geq 1\}$ alternatively. The word `alternatively' means that each partition in $P_{p_1,p_2}$ is composited by the members from $P_1$ and $P_2$ alternatively. That is,
    \begin{equation*}
        P_{p_1,p_2}(n)=S_1(n)\cup S_2(n) \cup S_3(n)\cup S_4(n),
    \end{equation*}
    where 
    \begin{align*}
        S_1(n)=&\bigcup_{m\geq 1}\left\{(a_1,b_1,a_2,b_2,...,a_m): a_i\in P_1, b_i\in P_2, \sum_{i=1}^m a_i+\sum_{j=1}^{m-1} b_j=n\right\},\\
        S_2(n)=&\bigcup_{m\geq 1}\left\{(a_1,b_1,a_2,b_2,...,a_m): a_i\in P_2, b_i\in P_1, \sum_{i=1}^m a_i+\sum_{j=1}^{m-1} b_j =n\right\},\\
        S_3(n)=&\bigcup_{m\geq 1}\left\{(a_1,b_1,a_2,b_2,...,a_m,b_m): a_i\in P_1, b_i\in P_2, \sum_{i=1}^m (a_i+ b_i)=n\right\},\\
        S_4(n)=&\bigcup_{m\geq 1}\left\{(a_1,b_1,a_2,b_2,...,a_m,b_m): a_i\in P_2, b_i\in P_1, \sum_{i=1}^m (a_i+ b_i)=n\right\}.
    \end{align*}
    Note that the dissimilarity between $S_1(n)$ and $S_2(n)$ (resp. $S_3(n)$ and $S_4(n)$) involves swapping the positions of $P_1$ and $P_2$.
\end{definition}

The following lemma estimates the cardinality of $\Delta^{(1,1)}_m$ (and further $\Delta_{m}^{(i,j)}$) by using the cardinality of $P_{p_1,p_2}(k)$. We remark that the cardinality of $P_{p_1,p_2}(k)$ is closely related to the classical integer partition problem which studies the asymptotic formula of the number of compositions of $n$ into $\mathbb{N}$ \cite{hardy1918aymp} (see also \cite[Chapter 4]{andrews1976}).

\begin{lemma}\label{lem 3-1}
 For $m\geq 0$ and for any $(i,j)$ with $i\geq 1$, $j\geq 1$, and at least one of $i,j$ equal to $1$,
    \begin{equation}\label{eq---1}
    \sum_{k=1}^{m+1}B^k\geq 1+\sum_{k=1}^m \left|P_{p_1,p_2}(k)\right|=\left|\Delta_{m}^{(1,1)}\right|\geq \left|\Delta_{m}^{(i,j)}\right|,
\end{equation}
where $B=\frac{1}{|x_0|}= 1.80193...$ and $x_0$ is the root of the polynomial $x^3-2x^2-x+1$ with the smallest absolute value.
\end{lemma}
\begin{proof}
We first prove the equality in the middle of (\ref{eq---1}). For any $m\geq 0$ and $k\geq 1$, define a map 
\[
\phi: \left\{g\in \Delta_m^{(1,1)}: |g|=k\right\}\to P_{p_1,p_2}(k)
\]
 by 
\begin{align*}
   \phi\left( f_1^{a_1}f_2^{a_2}\cdots \right)= \left(p_1^{\log_{p_1}(a_1+1)}-1,p_2^{\log_{p_2}(a_2+1)}-1,...\right),\\
    \phi\left(f_2^{b_1}f_1^{b_2}\cdots\right)=\left(p_2^{\log_{p_2}(b_1+1)}-1,p_1^{\log_{p_1}(b_2+1)}-1,...\right).
\end{align*}
To illustrate the definition of $\phi$, we can check that for the case $p_1=2$ and $p_2=3$, when $k=1$, we have
\begin{align*}
    \left\{g\in \Delta_m^{(1,1)}:  |g|=1\right\}=\left\{f_1\right\},~P_{2,3}(1)=\left\{(2^1-1)\right\}\mbox{ and }\phi(f_1)=(2^1-1).
\end{align*}
When $k=2$, we have
\begin{align*}
    \left\{g\in \Delta_m^{(1,1)}:  |g|=2\right\}=\left\{f_2^2\right\},~P_{2,3}(2)=\left\{(3^1-1)\right\}\mbox{ and } \phi(f_2^2)=(3^1-1).
\end{align*}
When $k=3$, we have
\begin{align*}
    &\left\{g\in \Delta_m^{(1,1)}:  |g|=3\right\}=\left\{f_1^3,f_1f_2^2,f_2^2f_1\right\},\\
    &P_{2,3}(3)=\left\{(2^2-1),(2-1,3-1),(3-1,2-1)\right\},\\
    &\phi(f_1^3)=(2^2-1),~\phi(f_1f_2^2)=(2-1,3-1)\mbox{ and }\phi(f_2^2f_1)=(3-1,2-1).
\end{align*}
See Figure \ref{fig 6} \begin{figure}
		\includegraphics[scale=0.55]{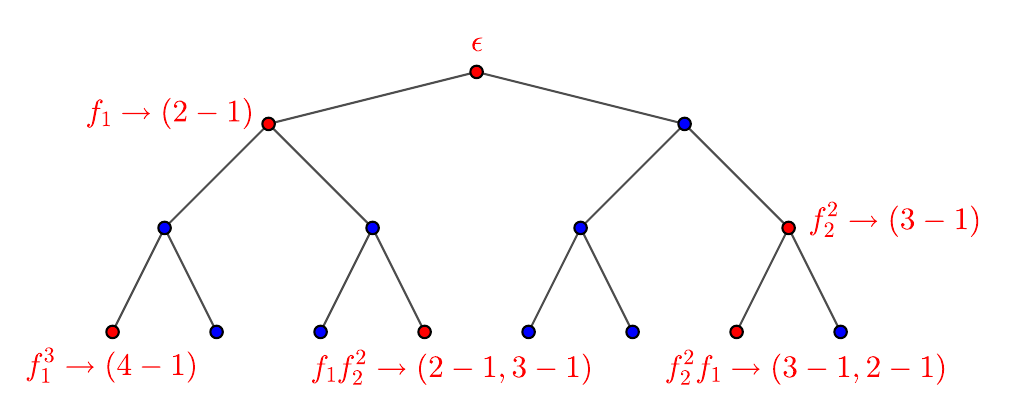}
		\caption{The map $\phi$ on $\Delta_3^{(1,1)}$ with respect to the multiplicative constraints $p_1=2$ and $p_2=3$.}\label{fig 6}
	\end{figure} for the map $\phi$ on $\Delta_3^{(1,1)}$ with respect to the multiplicative constraints $p_1=2$ and $p_2=3$. 
 
 It can also be checked that the map $\phi$ is a bijection with inverse map 
 \[\phi^{-1}: P_{p_1,p_2}(k)\mapsto \left\{g\in \Delta_m^{(1,1)}: |g|=k\right\}\]
 defined by
 \begin{align*}
     &\phi^{-1}\left(\left(p_1^{a_1}-1,p_2^{a_2}-1,...\right)\right)=f_1^{p_1^{a_1}-1}f_2^{p_2^{a_2}-1}\cdots,\\
      &\phi^{-1}\left(\left(p_2^{b_1}-1,p_1^{b_2}-1,...\right)\right)=f_2^{p_2^{b_1}-1}f_1^{p_1^{b_2}-1}\cdots.
 \end{align*}
Then, 
\begin{equation*}
    \left|\left\{g\in \Delta_m^{(1,1)}:  |g|=k\right\}\right|=\left|P_{p_1,p_2}(k)\right|.
\end{equation*}
Hence,
\begin{equation*}
    \left|\Delta_m^{(1,1)}\right|=|\{\epsilon\}|+\sum_{k=1}^m \left|\left\{g\in \Delta_m^{(1,1)}:  |g|=k\right\}\right|=1+\sum_{k=1}^m P_{p_1,p_2}(k).
\end{equation*}

For the left inequality of (\ref{eq---1}), we observe that for $p\geq 2$, the set $P_p(n)$ of ordered partitions of $n$ consists of elements from $\{1\}$ and the set $P_p^+:=\{p^n:n\geq 1\}$. That is,
\begin{equation*}
    P_p(n)=\bigcup_{\ell\geq 1}\left\{(a_1,a_2,...,a_{\ell}): a_i\in P_p^+\cup\{1\},~\sum_{i=1}^\ell a_i=n \right\}.
\end{equation*}

We first claim that for any $p_1\leq p_2$, 
\begin{equation}\label{eq 12}
    \left|P_{p_1}(n)\right|\geq \left|P_{p_2}(n)\right|.
\end{equation}
Since for any $x\in P_{p_2}(n)$, there exists an $\ell'\geq 1$ such that $x=(a_1,...,a_{\ell'})$ and $\sum_{i=1}^{\ell'} a_i=n$ where $a_i=(p_2)^{n_i},n_i\geq 0$ for all $1\leq i\leq \ell'$. Now, we find a member $y$ in $P_{p_1}(n)$ corresponding to $x$. That is, 
\[
y=\left((p_1)^{n_1},...,(p_1)^{n_{\ell'}},\overbrace{1,...,1}^{\ell''}\right),
\]
where $\ell''=\sum_{i=1}^{\ell'}(p_2^{n_i}-p_1^{n_i})$. Thus, (\ref{eq 12}) is obtained. Hence, 
\begin{equation}\label{eq 13}
    \left|P_2(n)\right|\geq \left|P_{p}(n)\right| \mbox{ for all }p\geq 2.
\end{equation}

We then claim that for any $p_1,p_2\geq 2$ 
\begin{equation}\label{eq 14}
    \left|P_2(n)\right|\geq \left|P_{p_1,p_2}(n)\right|.
\end{equation}
Since for any $x\in P_{p_1,p_2}(n)$, $x$ is in one of the sets $S_1(n),S_2(n),S_3(n)$ and $S_4(n)$. Without loss of generality, we may assume $x\in S_1(n)$. Then, there exits an $\ell\geq 1$ such that
\begin{equation*}
    x=\left(a_1,b_1,...,a_{\ell}\right) \mbox{ with }\sum_{i=1}^{\ell}a_i+\sum_{j=1}^{\ell-1}b_j=n 
\end{equation*}
where $ a_i=p_1^{n_i}-1,n_i\geq 1$ for all $1\leq i \leq \ell$, and $ b_j=p_2^{m_j}-1,m_j\geq 1$ for all $1\leq j \leq \ell-1$. Define
\begin{equation*}
    y=\left(2^{n_1-1},2^{m_1-1},...,2^{n_{\ell}-1},\overbrace{1,...,1}^{\ell'''}\right),
\end{equation*}
where 
\[
\ell'''=\sum_{i=1}^{\ell}\left(a_i-2^{n_i-1}\right)+\sum_{j=1}^{\ell-1}\left(b_j-2^{m_j-1}\right)\geq 0.
\]
Then, $y\in P_2(n)$. Hence, (\ref{eq 14}) holds.

Now, we show that $\left|P_2(n)\right|= O(B^n)$ with some $B<2$. Since the generating function of the sequence $\left\{|P_2(n)|\right\}_{n=1}^\infty$ is given by
\begin{equation*}
    \sum_{n=0}^\infty |P_2(n)| x^n=\frac{1}{1-\sum_{i=0}^\infty x^{2^i}}.
\end{equation*}
Then, by the fact that $\{2^i:i\geq 0\}\subseteq (\{1,2\}\cup\{4+2i: i\geq 0\})$, it can be checked that if the sequence $\{a(n)\}_{n=1}^\infty$ satisfies the generating function
\begin{equation}\label{eq 15}
     \sum_{n=0}^\infty a(n) x^n=\frac{1}{1-\left(x+x^2+x^4\sum_{i=0}^\infty x^{2i}\right)},
\end{equation}
then for all $n\geq 1$ 
\begin{equation}\label{eq15-1}
   |P_2(n)|\leq a(n).
\end{equation}
Let us compute $a(n)$. Note that
\begin{align*}
    \sum_{n=0}^\infty a(n) x^n=&\frac{1}{1-\left(x+x^2+x^4\sum_{i=0}^\infty x^{2i}\right)}\\
    =&\frac{1}{1-x-2x^2+x^3}\\
    =&\frac{A_1}{x-x_1}+\frac{A_2}{x-x_2}+\frac{A_3}{x-x_3}\\
    =&\frac{-A_1}{x_1}\sum_{n=0}^\infty\left(\frac{x}{x_1}\right)^n+\frac{-A_2}{x_2}\sum_{n=0}^\infty\left(\frac{x}{x_2}\right)^n+\frac{-A_3}{x_3}\sum_{n=0}^\infty\left(\frac{x}{x_3}\right)^n,
\end{align*}
where $A_1,A_2,A_3$ are constants and $x_1,x_2,x_3$ are the roots of $x^3-2x^2-x+1$ with $x_1=-0.80194...$, $x_2=0.55496...$ and $x_3=2.2470...$. Then, 
\begin{equation}\label{eq 16}
    a(n)=O\left(B^n\right),
\end{equation}
where $B=\frac{1}{|x_2|}= 1.80193...<2$.

Hence, by (\ref{eq15-1}) and (\ref{eq 16}), we obtain
\begin{equation}\label{eq 16-1}
    \left|P_2(n)\right|=O(B^n).
\end{equation}
Finally, by (\ref{eq 14}) and (\ref{eq 16-1}), we have 
\begin{equation*}
    1+\sum_{k=1}^m \left|P_{p_1,p_2}(k)\right|\leq \sum_{k=1}^{m+1} B^k~ (\forall m\geq 0).
\end{equation*}

Since the right inequality of (\ref{eq---1}) is clear, the proof is completed.
\end{proof}

\begin{remark}
    We remark that the asymptotic behavior of $\{|P_2(n)|\}_{n=1}^\infty$ is the sequence A023359 in OEIS \cite{A023359}. Cloitre and Kotesovec have a suspicion that
\begin{equation}\label{eq 16-0}
     \left|P_2(n)\right|= O(B^n),
\end{equation}
where $B=1.766398...$. Note that for the existence of entropy, we need only $B<2$ (which can also be obtained by (\ref{eq 16-0})). However, the precise surface entropy in Section \ref{sec4} will benefit from determining the asymptotic behavior rate of $P_2(n)$ (or $P_{p_1,p_2}(n)$).
\end{remark}

The next lemma gives the decomposition of $\Delta_n$ of the 2-tree.
\begin{lemma}\label{lem3-2}
    Let $p_1,p_2\geq 2$. For $n\geq 0$,
    \begin{align*}
         \Delta_n=&\left[\bigsqcup_{k=0}^{n-1}\left( \bigsqcup_{\substack{i+k=k+1\\ 1<i\in\mathcal{I}_{p_1}}}^{n} 2^{n-i-k}\Delta^{(i,1)}_k\right)\right]
       \bigsqcup\left[\bigsqcup_{k=0}^{n-1}\left( \bigsqcup_{\substack{j+k=k+1\\1<j\in\mathcal{I}_{p_2}}}^{n} 2^{n-j-k}\Delta^{(1,j)}_k\right)\right]\\
&\bigsqcup\left(\bigsqcup_{\substack{i=1\\1<i\in\mathcal{I}_{p_1}}}^{n+1}\Delta^{(i,1)}_{n+1-i}\right)\bigsqcup\left(\bigsqcup_{\substack{j=1\\1<j\in\mathcal{I}_{p_2}}}^{n+1}\Delta^{(1,j)}_{n+1-j}\right)\bigsqcup\Delta^{(1,1)}_n.
    \end{align*}
\end{lemma}

\begin{proof}
    Observe that for $n\geq 0$,
    \begin{align*}
        \Delta_n=&\left[\bigsqcup_{k=0}^{n-1} \left(\bigsqcup_{\substack{i<n+1-k\\1<i\in\mathcal{I}_{p_1}}}2^{n-i-k}\Delta^{(i,1)}_k\right)\right]
        \bigsqcup\left[\bigsqcup_{k=0}^{n-1} \left(\bigsqcup_{\substack{j<n+1-k\\1<j\in\mathcal{I}_{p_2}}}2^{n-j-k}\Delta^{(1,j)}_k\right)\right]\\
       &\bigsqcup\left(\bigsqcup_{\substack{i=1\\1<i\in\mathcal{I}_{p_1}}}^{n+1}\Delta^{(i,1)}_{n+1-i}\right)\bigsqcup\left(\bigsqcup_{\substack{j=1\\1<j\in\mathcal{I}_{p_2}}}^{n+1}\Delta^{(1,j)}_{n+1-j}\right)\bigsqcup\Delta^{(1,1)}_n.
  \end{align*}
Then, we have
  \begin{align*}
       \Delta_n=&\left[\bigsqcup_{k=0}^{n-1} \left(\bigsqcup_{\substack{i+k\leq n\\1<i\in\mathcal{I}_{p_1}}}2^{n-i-k}\Delta^{(i,1)}_k\right)\right]
       \bigsqcup\left[\bigsqcup_{k=0}^{n-1} \left(\bigsqcup_{\substack{j+k\leq n\\1<j\in\mathcal{I}_{p_2}}}2^{n-j-k}\Delta^{(1,j)}_k\right)\right]\\
       &\bigsqcup\left(\bigsqcup_{\substack{i=1\\1<i\in\mathcal{I}_{p_1}}}^{n+1}\Delta^{(i,1)}_{n+1-i}\right)\bigsqcup\left(\bigsqcup_{\substack{j=1\\1<j\in\mathcal{I}_{p_2}}}^{n+1}\Delta^{(1,j)}_{n+1-j}\right)\bigsqcup\Delta^{(1,1)}_n\\
       =&\left[\bigsqcup_{k=0}^{n-1}\left( \bigsqcup_{\substack{i+k=k+1\\1<i\in\mathcal{I}_{p_1}}}^n 2^{n-i-k}\Delta^{(i,1)}_k\right)\right]
       \bigsqcup\left[\bigsqcup_{k=0}^{n-1}\left( \bigsqcup_{\substack{j+k=k+1\\1<j\in\mathcal{I}_{p_2}}}^n 2^{n-j-k}\Delta^{(1,j)}_k\right)\right]\\
&\bigsqcup\left(\bigsqcup_{\substack{i=1\\1<i\in\mathcal{I}_{p_1}}}^{n+1}\Delta^{(i,1)}_{n+1-i}\right)\bigsqcup\left(\bigsqcup_{\substack{j=1\\1<j\in\mathcal{I}_{p_2}}}^{n+1}\Delta^{(1,j)}_{n+1-j}\right)\bigsqcup\Delta^{(1,1)}_n.
    \end{align*}
    See Figure \ref{fig 5} for the decomposition of $\Delta_3$ of the 2-tree with $p_1=2$ and $p_2=3$, i.e., 
    \begin{align*}
        \Delta_3=&\Delta_0^{(3,1)}
        \bigsqcup\Delta_0^{(1,2)}\sqcup\Delta_0^{(1,2)}\sqcup \Delta_1^{(1,2)}\\
        &\bigsqcup \Delta_1^{(3,1)}\bigsqcup \left(\Delta_2^{(1,2)}\sqcup \Delta_0^{(1,4)}\right)\bigsqcup \Delta_3^{(1,1)}.
    \end{align*}
     \begin{figure}
		\includegraphics[scale=0.6]{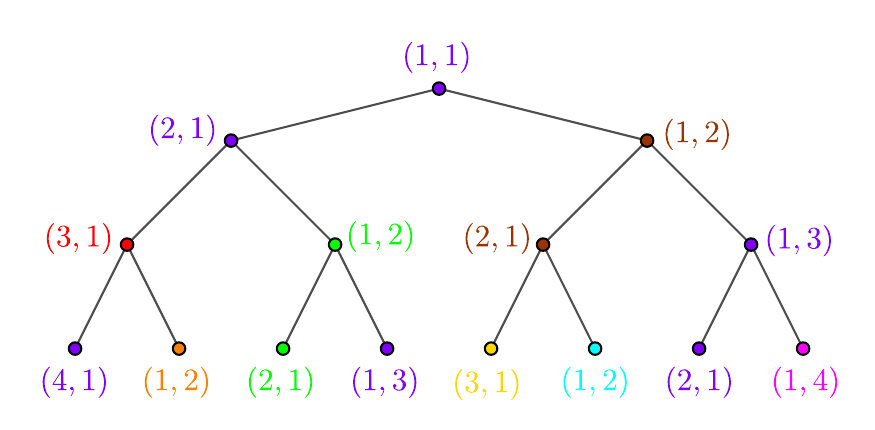}
		\caption{The decomposition of $\Delta_3$ of the 2-tree with $p_1=2$ and $p_2=3$.}\label{fig 5}
\end{figure}
The proof is complete.
\end{proof}

Now, we are ready to give the entropy formula of $X_{\Omega_1}^{p_1}\times X_{\Omega_2}^{p_2}$.

\begin{theorem}\label{thm 3-1}
The entropy of the axial product of $X_{\Omega_1}^{p_1}$ and $X_{\Omega_2}^{p_2}$ on the $2$-tree is 
\begin{align*}
    h\left( X_{\Omega_1}^{p_1}\times X_{\Omega_2}^{p_2}\right)=&\sum_{k=0}^{\infty}\sum_{\substack{i+k=k+1\\1<i\in\mathcal{I}_{p_1}}}^\infty\frac{1}{2^{k+i+1}}\log\left|\mathcal{P}\left(\Delta^{(i,1)}_k, \Omega_1\times \Omega_2\right)\right|\\
    &+\sum_{k=0}^{\infty}\sum_{\substack{j+k=k+1\\1<j\in\mathcal{I}_{p_2}}}^{\infty}\frac{1}{2^{k+j+1}}\log\left|\mathcal{P}\left(\Delta^{(1,j)}_k, \Omega_1\times \Omega_2\right)\right|,
\end{align*}
where $\mathcal{I}_{p_i}=\left\{j\in\mathbb{N}: p_i\nmid j\right\}$ for $i=1,2$, and $\mathcal{P}(\cdot,\cdot)$ is defined as in (\ref{P}).
\end{theorem}

\begin{proof}
 For $n\geq 0$, by Lemma \ref{lem3-2}, 
 \begin{align*}
&\left|\mathcal{P}\left(\Delta_n, X_{\Omega_1}^{p_1}\times X_{\Omega_2}^{p_2} \right)\right|\\
=&\left[\prod_{k=0}^{n-1}\left( \prod_{\substack{i+k=k+1\\1<i\in\mathcal{I}_{p_1}}}^n S(i,1,k)^{2^{n-i-k}}\right)\right]\times\left[\prod_{k=0}^{n-1}\left( \prod_{\substack{j+k=k+1\\1<j\in\mathcal{I}_{p_2}}}^n S(1,j,k)^{2^{n-j-k}}\right)\right]\\
&\times\left(\prod_{\substack{i=1\\1<i\in\mathcal{I}_{p_1}}}^{n+1}S(i,1,n+1-i)\right)\times\left(\prod_{\substack{j=1\\1<j\in\mathcal{I}_{p_2}}}^{n+1}S(1,j,n+1-j)\right)\times S(1,1,n),
 \end{align*}
 where 
\[
 S(a,b,c)=\left|\mathcal{P}\left(\Delta^{(a,b)}_c, \Omega_1\times \Omega_2\right)\right|.
\]
Then, we have
    \begin{align*}
         &\log\left|\mathcal{P}\left(\Delta_n, X_{\Omega_1}^{p_1}\times X_{\Omega_2}^{p_2} \right)\right|\\
         =&2^{n+1}\sum_{k=0}^{n-1}\sum_{\substack{i+k=k+1\\1<i\in\mathcal{I}_{p_1}}}^n \frac{\log S(i,1,k)}{2^{i+k+1}}+2^{n+1}\sum_{k=0}^{n-1}\sum_{\substack{j+k=k+1\\1<j\in\mathcal{I}_{p_2}}}^n \frac{\log S(1,j,k)}{2^{j+k+1}}\\
         &+\sum_{\substack{i=1\\1<i\in\mathcal{I}_{p_1}}}^{n+1}\log S(i,1,n+1-i)+ \sum_{\substack{j=1\\1<j\in\mathcal{I}_{p_2}}}^{n+1}\log S(1,j,n+1-j)+ \log S(1,1,n).
\end{align*}
Since $|\Delta_n|=2^{n+1}-1$, $\frac{2^{n+1}}{2^{n+1}-1}\leq 2~(\forall n\geq 0) $ and $S(a,b,c)\leq |\mathcal{A}|^{\left|\Delta^{(a,b)}_{c}\right|}$, we have
\begin{align*}
 &\frac{\log\left|\mathcal{P}\left(\Delta_n, X_{\Omega_1}^{p_1}\times X_{\Omega_2}^{p_2} \right)\right|}{\left|\Delta_n\right|}\\
         \leq& 2\sum_{k=0}^{n-1}\sum_{i=1}^n \frac{\log \left|\mathcal{A}\right|^{\left|\Delta^{(i,1)}_k\right|}}{2^{i+k+1}}+2\sum_{k=0}^{n-1}\sum_{j=1}^n \frac{\log \left|\mathcal{A}\right|^{\left|\Delta^{(1,j)}_k\right|}}{2^{j+k+1}}\\
         &+ \frac{\sum_{i=1}^{n+1}\log \left|\,\mathcal{A}\right|^{\left|\Delta^{(i,1)}_{n+1-i}\right|}}{2^{n+1}-1}+ \frac{\sum_{j=1}^{n+1}\log\left|\mathcal{A}\right|^{\left|\Delta^{(1,j)}_{n+1-j}\right|}}{2^{n+1}-1}+ \frac{\log\left|\mathcal{A}\right|^{\left|\Delta^{(1,1)}_n\right|}}{2^{n+1}-1}.
         \end{align*}
Therefore, by Lemma \ref{lem 3-1}, we have            
         \begin{align*}
         &\frac{\log\left|\mathcal{P}\left(\Delta_n, X_{\Omega_1}^{p_1}\times X_{\Omega_2}^{p_2} \right)\right|}{\left|\Delta_n\right|}\\
         \leq&\frac{2B\log \left|\mathcal{A}\right|}{B-1}\left( \sum_{k=0}^{n-1}\sum_{i=1}^n \frac{B^{k+1}}{2^{i+k+1}}+\sum_{k=0}^{n-1}\sum_{j=1}^n \frac{B^{k+1}}{2^{j+k+1}}\right.\\
         &\left.+\frac{\sum_{i=1}^{n+1}\left|B^{n+2-i}\right|}{2^{n+1}-1}+\frac{\sum_{j=1}^{n+1}B^{n+2-j}}{2^{n+1}-1}+\frac{B^{n+1}}{2^{n+1}-1}\right).
    \end{align*}

    By (\ref{eq 16-0}), we have $B<2$. Hence $\sum_{k=0}^{n-1}\sum_{i=1}^n \frac{B^{k+1}}{2^{i+k+1}}$ and $\sum_{k=0}^{n-1}\sum_{j=1}^n \frac{B^{k+1}}{2^{j+k+1}}$ converge as $n\to\infty$, and $\frac{\sum_{i=1}^{n+1}\left|B^{n+2-i}\right|}{2^{n+1}-1}$, $\frac{\sum_{j=1}^{n+1}B^{n+2-j}}{2^{n+1}-1}$ and $\frac{B^{n+1}}{2^{n+1}-1}$ tend to 0 as $n\to\infty$.

    Therefore, we have
    \begin{align*}
         h\left( X_{\Omega_1}^{p_1}\times X_{\Omega_2}^{p_2}\right)
        =&\lim_{n\to\infty}\frac{\left|\mathcal{P}\left(\Delta_n, X_{\Omega_1}^{p_1}\times X_{\Omega_2}^{p_2} \right)\right|}{\left|\Delta_n\right|}\\
          =&\lim_{n\to\infty}\frac{2^{n+1}}{2^{n+1}-1}\sum_{k=0}^{n-1}\sum_{\substack{i+k=k+1\\1<i\in\mathcal{I}_{p_1}}}^n \frac{\log S(i,1,k)}{2^{i+k+1}}\\
          &+\lim_{n\to\infty}\frac{2^{n+1}}{2^{n+1}-1}\sum_{k=0}^{n-1}\sum_{\substack{j+k=k+1\\1<j\in\mathcal{I}_{p_2}}}^n \frac{\log S(1,j,k)}{2^{j+k+1}}\\
         =&\sum_{k=0}^\infty\sum_{\substack{i+k=k+1\\1<i\in\mathcal{I}_{p_1}}}^\infty \frac{\log S(i,1,k)}{2^{i+k+1}}+\sum_{k=0}^\infty\sum_{\substack{j+k=k+1\\1<j\in\mathcal{I}_{p_2}}}^\infty \frac{\log S(1,j,k)}{2^{j+k+1}}.
    \end{align*}
   Recalling the definition of $S(a,b,k)$, we finish the proof.
    \end{proof}
    
 \begin{example}
    Let $p_1,p_2\geq 2$ and $\Omega_1=\Omega_2=\Sigma_A$ with $A=\left[\begin{matrix}
        0&1&1\\
        1&0&0\\
        1&0&0
    \end{matrix}\right]$. In order to compute the number $S(i,1,k)$ and $S(1,j,k)$, we need more information about $\Delta^{(1,j)}_k$ and $\Delta^{(i,1)}_k$. In fact, $\Delta_k^{(i,1)}$ is a union of $\Delta_\ell^{(1,p_2)}$'s, and $\Delta_k^{(1,j)}$ is a union of $\Delta_\ell^{(p_1,1)}$'s (see Figure \ref{fig 10}\begin{figure}
		\includegraphics[scale=0.4]{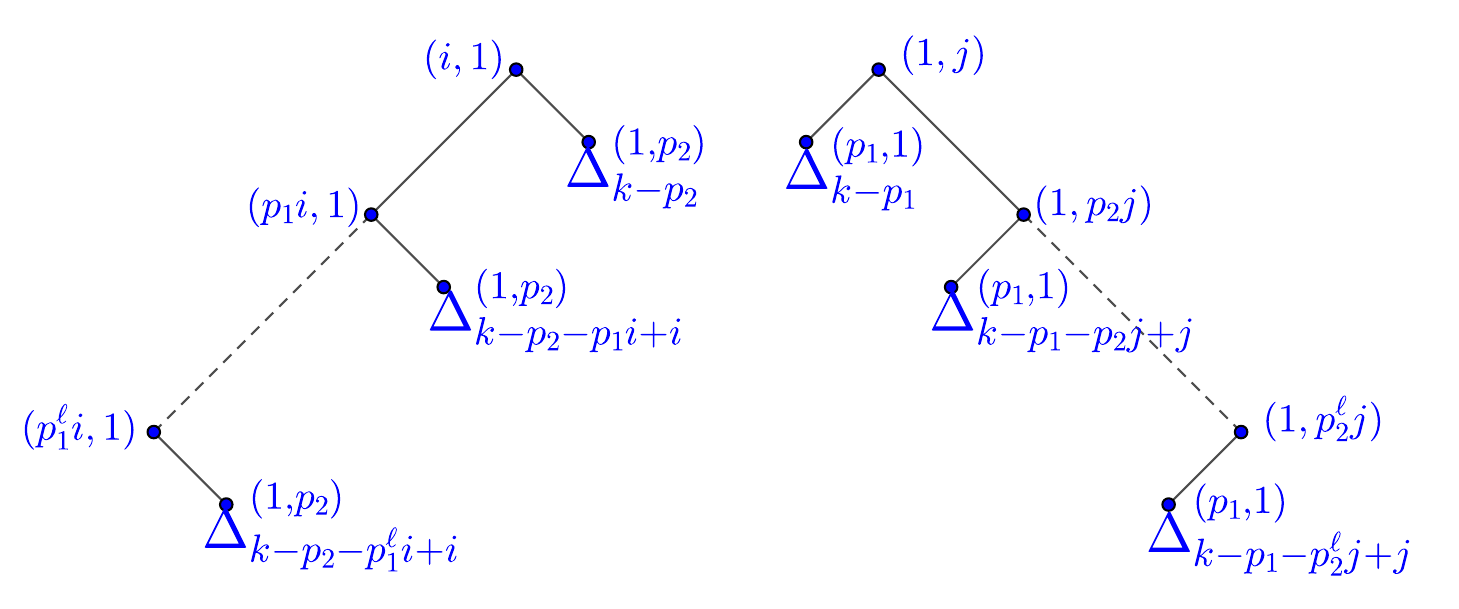}
		\caption{The graph of $\Delta_k^{(i,1)}$ (resp. $\Delta_k^{(1,j)}$), where $\ell =\left\lfloor \log_{p_1}\frac{i+k}{i}\right \rfloor $ (resp. $\ell =\left\lfloor \log_{p_2}\frac{j+k}{j}\right \rfloor$).}\label{fig 10}
\end{figure} for the graphs of $\Delta_k^{(i,1)}$ and $\Delta_k^{(1,j)}$). Thus, we need further informations on $\Delta_k^{(1,p_2)}$ and $\Delta_k^{(p_1,1)}$ for $k\geq 0$. Let $P^{(1)}(k+p_1)$ be the set of ordered partitions of $k+p_1$ into $\{p_1^n-1:n\geq 1\}$ and $\{p_2^n-1:n\geq 1\}$ alternatively with starting number $p_1^n-1$ for some $n\geq 1$. We establish a mapping with $p_1^n-1\mapsto n$ and $p_2^n-1\mapsto n$. If the sum of ordered partition equals $m$ after this mapping, it belongs to the $m$-th level of $\Delta_k^{(p_1,1)}$. We denote the size of the $m$-th level of $\Delta_k^{(p_1,1)}$ as $b^{(p_1)}_{k,m}$. Similarly, the size of the $m$-th level of $\Delta_k^{(1,p_2)}$ is denoted as $b^{(p_2)}_{k,m}$. Thus, we have
\begin{align*}
&S(i,1,k)\\
=&2^{\sum\limits_{m=0, m:odd}^{k-p_2} b^{(p_2)}_{k-p_2,m}+\sum\limits_{m=0, m:even}^{k-p_2-p_1i+i} b^{(p_2)}_{k-p_2-p_1i+i,m}+\cdots+\sum\limits_{m=0, m:even/odd}^{k-p_2-p_1^\ell i+i} b^{(p_2)}_{k-p_2-p_1^\ell i+i,m}}\\
&+2^{\sum\limits_{m=0, m:even}^{k-p_2} b^{(p_2)}_{k-p_2,m}+\sum\limits_{m=0, m:odd}^{k-p_2-p_1 i+i} b^{(p_2)}_{k-p_2-p_1i+i,m}+\cdots+\sum\limits_{m=0, m:odd/even}^{k-p_2-p_1^\ell i+i} b^{(p_2)}_{k-p_2-p_1^\ell i+i,m}},
\end{align*}
and
\begin{align*}
&S(1,j,k)\\
=&2^{\sum\limits_{m=0, m:odd}^{k-p_1} b^{(p_1)}_{k-p_1,m}+\sum\limits_{m=0, m:even}^{k-p_1-p_2j+j} b^{(p_1)}_{k-p_1-p_2j+j,m}+\cdots+\sum\limits_{m=0, m:even/odd}^{k-p_1-p_2^\ell j+j} b^{(p_1)}_{k-p_1-p_2^\ell j+j,m}}\\
&+2^{\sum\limits_{m=0, m:even}^{k-p_1} b^{(p_1)}_{k-p_1,m}+\sum\limits_{m=0, m:odd}^{k-p_1-p_2j+j} b^{(p_1)}_{k-p_1-p_2j+j,m}+\cdots+\sum\limits_{m=0, m:odd/even}^{k-p_1-p_2^\ell j+j} b^{(p_1)}_{k-p_1-p_2^\ell j+j,m}},
\end{align*}
where $even/odd$ is odd if $\ell$ is even, and $even/odd$ is even if $\ell$ is odd; $odd/even$ is odd if $\ell$ is odd, and $odd/even$ is even if $\ell$ is even.    
\end{example}

\subsection{Entropy formula for $X_{\Omega}^p \times X$}
Before we state the following results, define $\Delta_k^{(i,1)}$ with respect to the multiplicative constraint $p$ on the left to be the reshaped set of the subset of $g\Delta_k:=\{gg':g'\in \Delta_k\}$ such that $g$ is a vertex of type $(i,1)$ and each vertex in this subset is dependent on $g$ with respect to the multiplicative constraint $p$ on the left. (See Figure \ref{fig 3} \begin{figure}
		\includegraphics[scale=0.55]{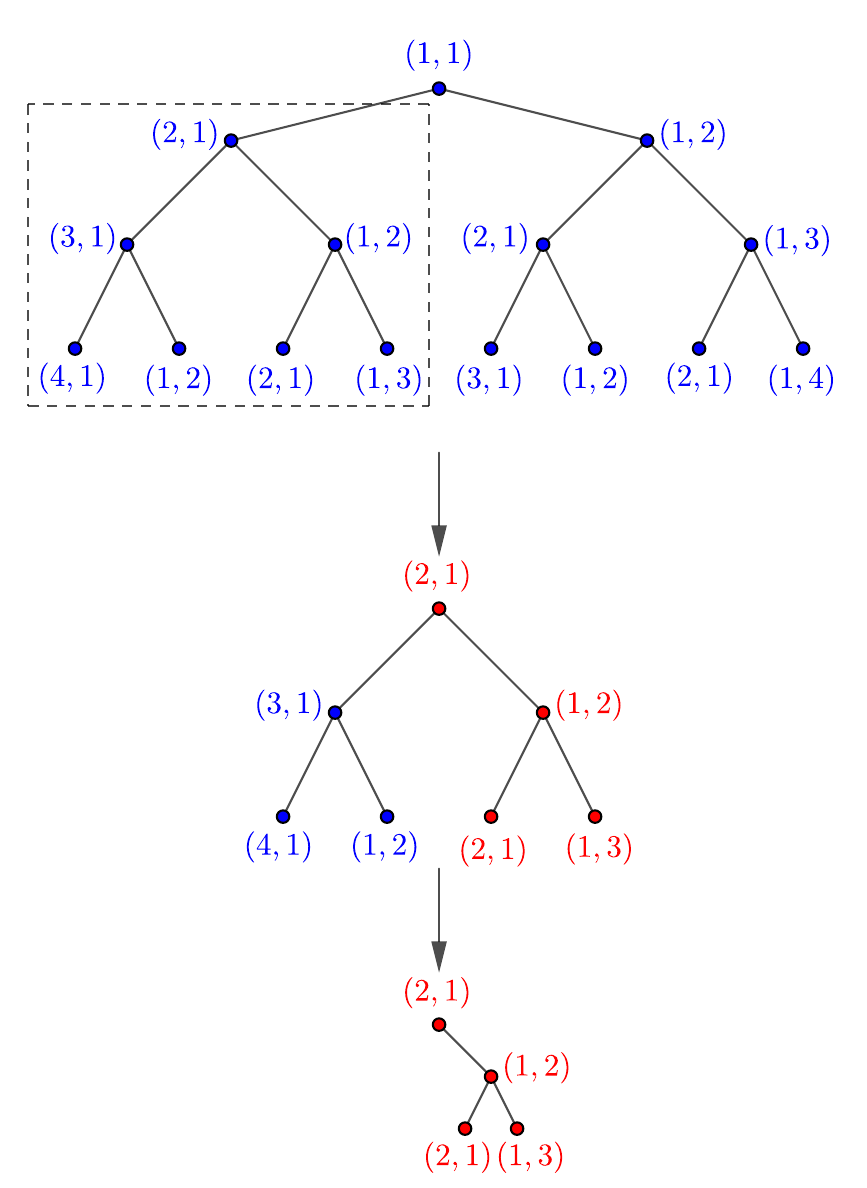}
		\caption{The $\Delta_2^{(2,1)}$ with respect to the multiplicative constraint $p=2$ on the left.}\label{fig 3}
\end{figure} for the graph of $\Delta_2^{(2,1)}$ with respect to multiplicative constraint $p=2$ on the left.)   

Recall that for any $n\geq 1$ and $p\geq 2$, $P_p(n)$ is the set of ordered partitions of $n$ consisting of elements from $\{1\}$ and $P_p^+=\{p^k:k\geq 1\}$.

\begin{lemma}\label{lem 3-4} 
 For $m\geq 0$ and for any $i\geq 1$, 
    \begin{equation}\label{eq--2}
    \sum_{k=1}^{m+1}B^k\geq\sum_{k=1}^{m+1} P_p(k)=\left|\Delta_{m}^{(1,1)}\right|\geq \left|\Delta_{m}^{(i,1)}\right|.
\end{equation}
\end{lemma}

\begin{proof}
  We first prove the equality in the middle of (\ref{eq--2}). For any $m\geq 0$ and $k\geq 1$, define a map 
\[
\psi: \left\{g\in \Delta_m^{(1,1)}: |g|=k-1\right\}\to P_p(k)
\]
 by 
\begin{align*}
\psi(\epsilon)&=(1),\\
\psi\left(f_1^{a_1}f_2^{b_1}f_1^{a_2}\cdots\right)&= \left((a_1+1),\overbrace{1,...,1}^{b_1-1},a_2+1,...\right),\\
\psi\left(f_2^{b_1}f_1^{a_1}\cdots \right)&=\left( \overbrace{1,...,1}^{b_1},(a_1+1),...\right), \\
\psi\left(\cdots f_1^{a_1}f_2^{b_1}\right)&=\left( ...,(a_1+1),\overbrace{1,...,1}^{b_1}\right), \\
\psi\left(f_2^{b_1}\right)&=\left( \overbrace{1,...,1}^{b_1+1}\right).
\end{align*}
 For example, if $p=2$ and $k=1$, 
\begin{align*}
    \left\{g\in \Delta_m^{(1,1)}:  |g|=0\right\}=\left\{\epsilon\right\},~P_2(1)=\left\{1\right\}\mbox{ and }\psi(\epsilon)=(1).
\end{align*}
For $k=2$,
\begin{align*}
    &\left\{g\in \Delta_m^{(1,1)}:  |g|=1\right\}=\left\{f_1,f_2\right\},~P_2(2)=\left\{(2),(1,1)\right\},\\
    &\psi(f_1)=(2)\mbox{ and }\psi(f_2)=(1,1).
\end{align*}
For $k=3$,
\begin{align*}
    &\left\{g\in \Delta_m^{(1,1)}: |g|=2\right\}=\left\{f_1f_2,f_2f_1,f_2^2\right\},\\
    &P_2(3)= \left\{(2,1),(1,2),(1,1,1)\right\},\\
    &\psi(f_1f_2)=(2,1),~\psi(f_2f_1)=(1,2)\mbox{ and }\psi(f_2^2)=(1,1,1).
\end{align*}
For $k=4$,
\begin{align*}
    &\left\{g\in \Delta_m^{(1,1)}:  |g|=k-1=3\right\}=
\left\{f_1^3,f_1f_2f_1,f_1f_2^2,f_2f_1f_2,f_2^2f_1,f_2^3\right\},\\
    &P_2(4)=\left\{(2^2),(2,2),(2,1,1),(1,2,1),(1,1,2),(1,1,1,1)\right\},\\
    &\psi(f_1^3)=(2^2),~\psi(f_1f_2f_1)=(2,2),~\psi(f_1f_2^2)=(2,1,1),~\psi(f_1f_1f_2)=(1,2,1),\\
    &\psi(f_2^2f_1)=(1,1,2)\mbox{ and }\psi(f_2^3)=(1,1,1,1).
\end{align*}
See Figure \ref{fig 7} \begin{figure}
		\includegraphics[scale=0.55]{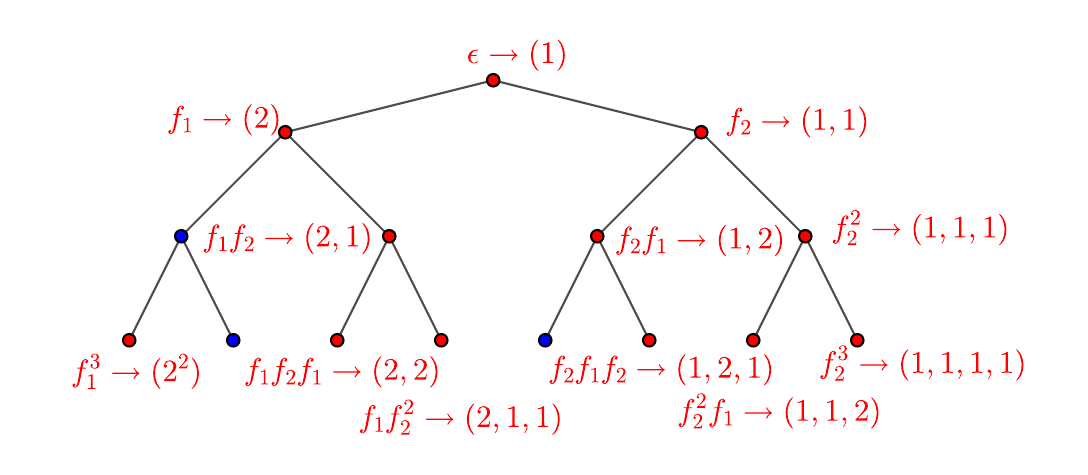}
		\caption{The map $\psi$ on $\Delta_3^{(1,1)}$ with respect to the multiplicative constraint $p=2$ on the left.}\label{fig 7}
	\end{figure} for the map $\psi$ on $\Delta_3^{(1,1)}$ with respect to the multiplicative constraint $p=2$ on the left.
 
 We can check that the map $\psi$ is a bijection with inverse map 
 \[\psi^{-1}: P_p(k)\to \left\{g\in \Delta_m^{(1,1)}: |g|=k-1\right\}\]
 defined by
 \begin{align*}
     &\psi^{-1}\left(\left(a_1,\overbrace{1,...,1}^{b_1},a_2,...\right)\right)=f_1^{a_1-1}f_2^{b_1+1}f_1^{a_2-1}\cdots,\\
     &\psi^{-1}\left(\left(\overbrace{1,...,1}^{b_1},a_1,...\right)\right)=f_2^{b_1}f_1^{a_1-1}\cdots,\\
      &\psi^{-1}\left(\left(...,a_1,\overbrace{1,...,1}^{b_1}\right)\right)=\cdots f_1^{a_1-1}f_2^{b_1},\\
      &\psi^{-1}\left(\left(\overbrace{1,...,1}^{b_1}\right)\right)=f_2^{b_2-1},
 \end{align*}
 where $a_1,a_2\in P_p^+$ and $b_1\geq 0,b_2\geq 1$. Then, 
\begin{equation*}
    \left|\left\{g\in \Delta_m^{(1,1)}:  |g|=k-1\right\}\right|=\left|P_p(k)\right|.
\end{equation*}
Thus,
\begin{equation*}
    \left|\Delta_m^{(1,1)}\right|=\sum_{k=1}^{m+1} \left|\left\{g\in \Delta_m^{(1,1)}:  |g|=k-1\right\}\right|=\sum_{k=1}^{m+1} \left|P_p(k)\right|.
\end{equation*}
 The inequalities of (\ref{eq--2}) follow by the same reason as in the proof of Lemma \ref{lem 3-1}. The proof is completed. 
 \end{proof}

The next lemma gives another decomposition of $\Delta_n$ of the 2-tree.
\begin{lemma}\label{lem3-5}
    Let $p\geq 2$. For $n\geq 0$,
    \begin{align*}
        \Delta_n=&\left[\bigsqcup_{k=0}^{n-1}\left( \bigsqcup_{\substack{i+k=k+1\\1<i\in\mathcal{I}_p}}^n 2^{n-i-k}\Delta^{(i,1)}_k\right)\right]
        \bigsqcup\left(\bigsqcup_{\substack{i=1\\1<i\in\mathcal{I}_p}}^{n+1}\Delta^{(i,1)}_{n+1-i}\right)\bigsqcup\Delta^{(1,1)}_n.
\end{align*}
\end{lemma}

\begin{proof}
    Observe that for $n\geq 0$,
    \begin{align*}
        \Delta_n=&\left[\bigsqcup_{k=0}^{n-1} \left(\bigsqcup_{\substack{i<n+1-k\\1<i\in\mathcal{I}_p}} 2^{n-i-k}\Delta^{(i,1)}_k\right)\right]
        \bigsqcup\left(\bigsqcup_{\substack{i=1\\1<i\in\mathcal{I}_p}}^{n+1}\Delta^{(i,1)}_{n+1-i}\right)\bigsqcup\Delta^{(1,1)}_n.
   \end{align*}
Clearly,
   \begin{align*}
        \Delta_n=&\left[\bigsqcup_{k=0}^{n-1}\left( \bigsqcup_{\substack{i+k=k+1\\1<i\in\mathcal{I}_p}}^n 2^{n-i-k}\Delta^{(i,1)}_k\right)\right]
        \bigsqcup\left(\bigsqcup_{\substack{i=1\\1<i\in\mathcal{I}_p}}^{n+1}\Delta^{(i,1)}_{n+1-i}\right)\bigsqcup\Delta^{(1,1)}_n.
\end{align*}
See Figure \ref{fig 4} for the decomposition of $\Delta_3$ with $p=2$, i.e., 
\begin{align*}
    \Delta_3=&\Delta_0^{(3,1)}
    \bigsqcup \Delta_1^{(3,1)}\bigsqcup \Delta_3^{(1,1)}. 
\end{align*}
\begin{figure}
		\includegraphics[scale=0.55]{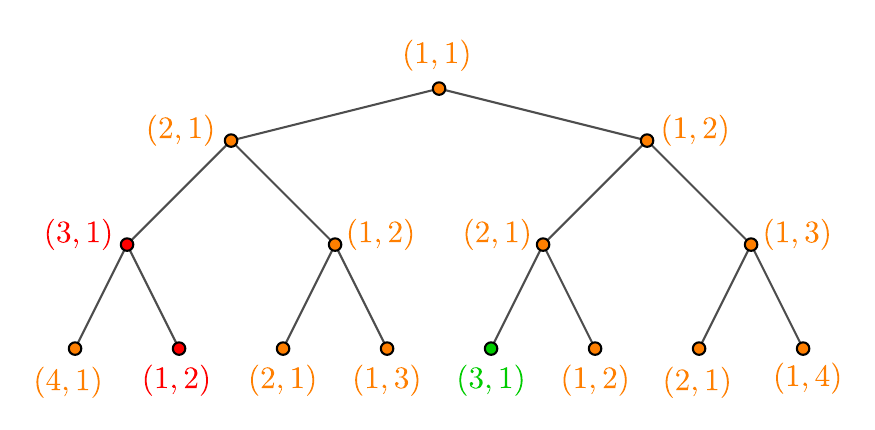}
		\caption{The decomposition of $\Delta_3$ with respect to $X_{\Omega}^p \times X $ with $p=2$.}\label{fig 4}
\end{figure} 
The proof is complete.
\end{proof}

\begin{theorem}\label{thm 3-2}
The entropy of the axial product of $X_{\Omega}^{p}$ and $X$ on the $2$-tree is 
\begin{align*}
    h\left(X_{\Omega}^p \times X \right)=&\sum_{k=0}^{\infty}\sum_{\substack{i+k=k+1\\1<i\in\mathcal{I}_p}}^\infty\frac{1}{2^{k+i+1}}\log\left|\mathcal{P}\left(\Delta^{(i,1)}_k, \Omega\times X\right)\right|,
\end{align*}
where $\mathcal{I}_p=\left\{i\in \mathbb{N}: p\nmid i\right\}$, and $\mathcal{P}(\cdot,\cdot)$ is defined as in (\ref{P}).
\end{theorem}
\begin{proof}
By Lemma \ref{lem3-5}, for $n\geq 0$, we have
 \begin{align*}
       &\left|\mathcal{P}\left(\Delta_n,X_{\Omega}^p \times X\right)\right| \\
       =&\left[\prod_{k=0}^{n-1}\left( \prod_{\substack{i+k=k+1\\1<i\in\mathcal{I}_p}}^n \left|\mathcal{P}\left(\Delta^{(i,1)}_k, \Omega\times X\right)\right|^{2^{n-i-k}}\right)\right]\\
&\times\left(\prod_{\substack{i=1\\1<i\in\mathcal{I}_p}}^{n+1}\left|\mathcal{P}\left(\Delta^{(i,1)}_{n+1-i}, \Omega\times X\right)\right|\right)
\times\left|\mathcal{P}\left(\Delta^{(1,1)}_n, \Omega\times X\right)\right|.
\end{align*}
Then, by Lemma \ref{lem 3-4} and the similar argument as in Theorem \ref{thm 3-1}, we complete the proof. 
\end{proof}

\section{Surface Entropy of axial products on $\mathbb{N}^2$ and the 2-tree}\label{sec4}
In this section, we study the surface entropies of $X_{\Omega_1}^{p_1} \otimes X_{\Omega_2}^{p_2}$, $X_{\Omega}^p \otimes X$, $X_{\Omega_1}^{p_1} \times X_{\Omega_2}^{p_2}$ and $X_{\Omega}^p \otimes X$.

\begin{theorem}\label{thm4.1}
We have the following assertions.
\begin{enumerate}
    \item We have $\log\left|\mathcal{P}\left(\left\llbracket (1,1),\left(m,n\right)\right\rrbracket,X_{\Omega_1}^{p_1} \otimes X_{\Omega_2}^{p_2} \right)\right|-mnh(X_{\Omega_1}^{p_1} \otimes X_{\Omega_2}^{p_2})=O(n\log m +m\log n)$.
    \item Suppose that for all $k\geq 1$, $\Omega\otimes X$ on $\left\llbracket (1,1),\left(k,\infty\right)\right\rrbracket$ is a mixing SFT with adjacency matrix $A_k$ having eigenvalue $\rho_k$ and normalized left $\ell_k$ and right $r_k$ eigenvectors satisfying $r_k \cdot \ell_k=1$ and $H(k)=\log \frac{\sum_i (\ell_k)_i \sum_{i}(r_k)_i}{\rho_k}$. Then, we have $\log\left|\mathcal{P}\left(\left\llbracket (1,1),\left(m,n\right)\right\rrbracket,X_{\Omega}^p \otimes X \right)\right|-mnh(X_{\Omega}^p \otimes X)=O(n\log m)+O\left(m\sum_{k=1}^\infty\frac{H(k)}{p^{k-1}}\right)$.
    \item We have $\log\left|\mathcal{P}\left(\Delta_n,X_{\Omega_1}^{p_1} \times X_{\Omega_2}^{p_2} \right)\right|-|\Delta_n|h(X_{\Omega_1}^{p_1} \times X_{\Omega_2}^{p_2})=O(B^n)$.
    \item We have $\log\left|\mathcal{P}\left(\Delta_n,X_{\Omega}^p \times X \right)\right|-|\Delta_n|h(X_{\Omega}^p \otimes X)=O(B^n)$.  
    \end{enumerate}
\end{theorem}
\begin{proof}
\item[\bf (1)]
For $m,n\geq 1$, we have 
\begin{align*}
        &\log\left|\mathcal{P}\left(\left\llbracket (1,1),\left(m,n\right)\right\rrbracket,X_{\Omega_1}^{p_1} \otimes X_{\Omega_2}^{p_2} \right)\right|\\
=&\sum_{i_1=1}^{\left\lfloor\log_{p_1}m\right\rfloor+1}\sum_{i_2=1}^{\left\lfloor\log_{p_2}n\right\rfloor+1}
g(p_1,m,i_1)h(p_1,m,i_1)g(p_2,n,i_2)h(p_2,n,i_2)P(i_1,i_2),
\end{align*}
where $P(a,b)=\log\left|\mathcal{P}\left(\left\llbracket (1,1),\left( a,b\right)\right\rrbracket,\Omega_1 \otimes \Omega_2 \right)\right|$,
\[g(a,b,c)=\left\lfloor\frac{b}{a^{c-1}}\right\rfloor-\left\lfloor\frac{b}{a^c}\right\rfloor\]
and
\[h(a,b,c)=\frac{\left|\mathcal{I}_{a}\cap \left(\left\lfloor\frac{b}{a^{c}}\right\rfloor,\left\lfloor\frac{b}{a^{c-1}}\right\rfloor\right]\right|}{\left\lfloor\frac{b}{a^{c-1}}\right\rfloor-\left\lfloor\frac{b}{a^{c}}\right\rfloor}.\]
Then,
\begin{align*}
&\log\left|\mathcal{P}\left(\left\llbracket (1,1),\left(m,n\right)\right\rrbracket,X_{\Omega_1}^{p_1} \otimes X_{\Omega_2}^{p_2} \right)\right|\\
=&\sum_{i_1=1}^{\left\lfloor\log_{p_1}m\right\rfloor+1}\sum_{i_2=1}^{\left\lfloor\log_{p_2}n\right\rfloor+1}\left\{
G(p_1,m,i_1,p_2,n,i_2)P(i_1,i_2)\right.\\
&+\frac{p_1-1}{p_1}\left(\frac{m}{p_1^{i_1-1}}-\frac{m}{p_1^{i_1}}\right)\frac{p_2-1}{p_2}\left[h(p_2,n,i_2)-\left(\frac{n}{p_2^{i_2-1}}-\frac{n}{p_2^{i_2}}\right)\right]
P(i_1,i_2)\\
&+\frac{p_1-1}{p_1}\left(\frac{m}{p_1^{i_1-1}}-\frac{m}{p_1^{i_1}}\right)\left[g(p_2,n,i_2)-\frac{p_2-1}{p_2}\right]h(p_2,n,i_2)P(i_1,i_2)\\
&+\frac{p_1-1}{p_1}\left[h(p_1,m,i_1)-\left(\frac{m}{p_1^{i_1-1}}-\frac{m}{p_1^{i_1}}\right)\right]g(p_2,n,i_2)h(p_2,n,i_2)P(i_1,i_2)\\
&\left.\left[g(p_1,m,i_1)-\frac{p_1-1}{p_1}\right]h(p_1,m,i_1)g(p_2,n,i_2)h(p_2,n,i_2)P(i_1,i_2)\right\},
\end{align*}
where
\[
G(a,b,c,d,e,f)=\left(1-\frac{1}{a}\right)\left(\frac{b}{a^{c-1}}-\frac{b}{a^{c}}\right)\left(1-\frac{1}{d}\right)\left(\frac{e}{d^{f-1}}-\frac{e}{d^{f}}\right).
\]

On the other hand, 
\begin{align*}
    &mnh(X_{\Omega_1}^{p_1} \otimes X_{\Omega_2}^{p_2})\\
    =&\sum_{i_1=1}^{\infty}\sum_{i_2=1}^{\infty}
G(p_1,m,i_1,p_2,n,i_2)P(i_1,i_2)\\
=&\sum_{i_1=1}^{\left\lfloor\log_{p_1}m\right\rfloor+1}\sum_{i_2=1}^{\left\lfloor\log_{p_2}n\right\rfloor+1}
G(p_1,m,i_1,p_2,n,i_2)P(i_1,i_2)\\
&+\sum_{i_1=\left\lfloor\log_{p_1}m\right\rfloor+2}^{\infty}\sum_{i_2=1}^{\left\lfloor\log_{p_2}n\right\rfloor+1}
G(p_1,m,i_1,p_2,n,i_2)P(i_1,i_2)\\
&+\sum_{i_1=1}^{\left\lfloor\log_{p_1}m\right\rfloor+1}\sum_{i_2=\left\lfloor\log_{p_2}n\right\rfloor+2}^{\infty}
G(p_1,m,i_1,p_2,n,i_2)P(i_1,i_2)\\
&+\sum_{i_1=\left\lfloor\log_{p_1}m\right\rfloor+2}^{\infty}\sum_{i_2=\left\lfloor\log_{p_2}n\right\rfloor+2}^{\infty}
G(p_1,m,i_1,p_2,n,i_2)P(i_1,i_2).
\end{align*}
Then, by the following two calculations:
    \begin{align*}
&\sum_{i_1=1}^{\left\lfloor\log_{p_1}m\right\rfloor+1}\sum_{i_2=1}^{\left\lfloor\log_{p_2}n\right\rfloor+1}
g(p_1,m,i_1)h(p_1,m,i_1)g(p_2,n,i_2)h(p_2,n,i_2)P(i_1,i_2)\\
\leq&\sum_{i_1=1}^{\left\lfloor\log_{p_1}m\right\rfloor+1}\sum_{i_2=1}^{\left\lfloor\log_{p_2}n\right\rfloor+1}
2\frac{n}{p_2^{i_2-1}}
i_1i_2\log\left|\mathcal{A}\right|\\
\leq& \left(\left\lfloor\log_{p_1}m\right\rfloor+1\right)2n\log\left|\mathcal{A}\right| \sum_{i_2=1}^\infty
\frac{i_2}{p_2^{i_2-1}}\\
=& O(n\log m)
    \end{align*}
    and
    \begin{align*}
&\sum_{i_1=\left\lfloor\log_{p_1}m\right\rfloor+2}^{\infty}\sum_{i_2=1}^{\left\lfloor\log_{p_2}n\right\rfloor+1}
G(p_1,m,i_1,p_2,n,i_2)P(i_1,i_2)\\
\leq & \sum_{i_1=\left\lfloor\log_{p_1}m\right\rfloor+2}^{\infty}\sum_{i_2=1}^{\left\lfloor\log_{p_2}n\right\rfloor+1}
\left(\frac{m}{p_1^{i_1-1}}-\frac{m}{p_1^{i_1}}\right)\left(\frac{n}{p_2^{i_2-1}}-\frac{n}{p_2^{i_2}}\right)
i_1i_2\log\left|\mathcal{A}\right|\\
\leq & mn\sum_{i_1=\left\lfloor\log_{p_1}m\right\rfloor+2}^{\infty}\frac{i_1}{p_1^{i_1-1}}\sum_{i_2=1}^\infty\frac{i_2}{p_2^{i_2-1}}\log\left|\mathcal{A}\right|\\
=&O(n),
    \end{align*}
    we have 
    \[\log\left|\mathcal{P}\left(\left\llbracket (1,1),\left(m,n\right)\right\rrbracket,X_{\Omega_1}^{p_1} \otimes X_{\Omega_2}^{p_2} \right)\right|-mnh(X_{\Omega_1}^{p_1} \otimes X_{\Omega_2}^{p_2})=O(n\log m +m\log n).\]
\item[\bf (2)]Note that for $m,n\geq 1$,
\begin{align*}
        &\log\left|\mathcal{P}\left(\left\llbracket (1,1),\left(m,n\right)\right\rrbracket,X_{\Omega}^p \otimes X \right)\right|=\sum_{k=1}^{\left\lfloor\log_{p}m\right\rfloor+1}g(p,m,k)h(p,m,k)P'(k,n),
\end{align*}
where $P'(a,b)=\log\left|\mathcal{P}\left(\left\llbracket (1,1),\left( a,b\right)\right\rrbracket,\Omega \otimes X \right)\right|$. Recall that $\Omega\otimes X$ on $\left\llbracket (1,1),\left(k,\infty\right)\right\rrbracket$ is a mixing SFT with adjacency matrix $A_k$ having eigenvalue $\rho_k$ and normalized left $\ell_k$ and right $r_k$ eigenvectors satisfying $r_k \cdot \ell_k=1$ for all $k\geq 1$. Using $\lambda_k:=\log \rho_k$, we write
\begin{align*}
&\log\left|\mathcal{P}\left(\left\llbracket (1,1),\left(m,n\right)\right\rrbracket,X_{\Omega}^p \otimes X \right)\right|\\
=&\sum_{k=1}^{\left\lfloor\log_{p}m\right\rfloor+1}\left[g(p,m,k)-\left(1-\frac{1}{p}\right)\right]h(p,m,k)P'(k,n)\\
&+\sum_{k=1}^{\left\lfloor\log_{p}m\right\rfloor+1}\left(1-\frac{1}{p}\right)\left[h(p,m,k)-\left(\frac{m}{p^{k-1}}-\frac{m}{p^k}\right)\right]P'(k,n)\\
&+\sum_{k=1}^{\left\lfloor\log_{p}m\right\rfloor+1}\left(1-\frac{1}{p}\right)\left(\frac{m}{p^{k-1}}-\frac{m}{p^k}\right)\left[P'(k,n)-n\lambda_k\right]\\
&+\sum_{k=1}^{\left\lfloor\log_{p}m\right\rfloor+1}\left(1-\frac{1}{p}\right)\left(\frac{m}{p^{k-1}}-\frac{m}{p^k}\right)n\lambda_k.
    \end{align*}
 By \cite[Theorem 3.1.2]{pace2018surface}, we obtain that for all $k\geq 1$,
\begin{align*}
  P'(k,n)-n\lambda_k=O\left(H(k)\right),
\end{align*}
where $H(k)=\log \frac{\sum_i (\ell_k)_i \sum_{i}(r_k)_i}{\rho_k}$.

Therefore,
        \begin{align*}
        &\log\left|\mathcal{P}\left(\left\llbracket (1,1),\left(m,n\right)\right\rrbracket,X_{\Omega}^p \otimes X \right)\right|-mnh(X_{\Omega}^p \otimes X)\\
=&O(n\log m)+O\left(m\sum_{k=1}^\infty\frac{H(k)}{p^{k-1}}\right)-mn \sum_{k=\left\lfloor\log_{p}m\right\rfloor+2}^\infty \frac{\lambda_i}{p^{i-1}}\\
\leq& O(n\log m)+O\left(m\sum_{k=1}^\infty\frac{H(k)}{p^{k-1}}\right)+mn \sum_{k=\left\lfloor\log_{p}m\right\rfloor+2}^\infty \frac{i\log\left|\mathcal{A}\right|}{p^{i-1}}\\
=&O(n\log m)+O\left(m\sum_{k=1}^\infty\frac{H(k)}{p^{k-1}}\right).
    \end{align*}
 \item[\bf (3)]
    Note that for $n\geq 0$,
    \begin{align*}
         &\log\left|\mathcal{P}\left(\Delta_n,X_{\Omega_1}^{p_1} \times X_{\Omega_2}^{p_2} \right)\right|\\
         =&\sum_{k=0}^{n-1}\sum_{\substack{i+k=k+1\\1<i\in\mathcal{I}_{p_1}}}^n 2^{n-i-k}\log S(i,1,k)+\sum_{k=0}^{n-1}\sum_{\substack{j+k=k+1\\1<j\in\mathcal{I}_{p_2}}}^n 2^{n-j-k}\log S(1,j,k)\\
          &+\sum_{\substack{i=1\\1<i\in\mathcal{I}_{p_1}}}^{n+1}\log S(i,1,n+1-i)+\sum_{\substack{j=1\\1<j\in\mathcal{I}_{p_2}}}^{n+1}\log S(1,j,n+1-j)
         + \log S(1,1,n)
    \end{align*}
    and
    \begin{align*}
        &|\Delta_n|h(X_{\Omega_1}^{p_1} \times X_{\Omega_2}^{p_2})\\
        =&-h(X_{\Omega_1}^{p_1} \times X_{\Omega_2}^{p_2})\\
        &+\sum_{k=0}^{\infty}\sum_{\substack{i+k=k+1\\1<i\in\mathcal{I}_{p_1}}}^\infty 2^{n-i-k}\log S(i,1,k)+\sum_{k=0}^{\infty}\sum_{\substack{j+k=k+1\\1<j\in\mathcal{I}_{p_2}}}^\infty 2^{n-j-k}\log S(1,j,k)\\
    =&-h(X_{\Omega_1}^{p_1} \times X_{\Omega_2}^{p_2})\\
    &+\sum_{k=0}^{n-1}\sum_{\substack{i+k=k+1\\1<i\in\mathcal{I}_{p_1}}}^n 2^{n-i-k}\log S(i,1,k)
    +\sum_{k=0}^{n-1}\sum_{\substack{j+k=k+1\\1<j\in\mathcal{I}_{p_2}}}^n 2^{n-j-k}\log S(1,j,k)\\
    &+\sum_{k=0}^{n-1}\sum_{\substack{i+k=n+1\\1<i\in\mathcal{I}_{p_1}}}^\infty 2^{n-i-k}\log S(i,1,k)
    +\sum_{k=0}^{n-1}\sum_{\substack{j+k=n+1\\1<j\in\mathcal{I}_{p_2}}}^\infty 2^{n-j-k}\log S(1,j,k)\\    &+\sum_{k=n}^{\infty}\sum_{\substack{i+k=k+1\\1<i\in\mathcal{I}_{p_1}}}^\infty 2^{n-i-k}\log S(i,1,k)
    +\sum_{k=n}^{\infty}\sum_{\substack{j+k=k+1\\1<j\in\mathcal{I}_{p_2}}}^\infty 2^{n-j-k}\log S(1,j,k).
    \end{align*}
  Then, by the following three calculations: 
  \begin{align*}
      &\sum_{k=0}^{n-1}\sum_{i+k=n+1,1<i\in\mathcal{I}_{p_1}}^\infty 2^{n-i-k}\log S(i,1,k) \\
      \leq& \sum_{k=0}^{n-1}\sum_{i+k=n+1,1<i\in\mathcal{I}_{p_1}}^\infty 2^{n-i-k}C_1 B^{k+1}\log\left|\mathcal{A}\right|\\
      =&O(B^n)
  \end{align*}
  and
  \begin{align*}
      &\sum_{k=n}^{\infty}\sum_{i+k=k+1,1<i\in\mathcal{I}_{p_1}}^\infty 2^{n-i-k}\log S(i,1,k)\\
      \leq& \sum_{k=n}^{\infty}C_1 B^{k+1}\log\left|\mathcal{A}\right|\sum_{i+k=k+1}^\infty 2^{n-(i+k)}\\
      =&\sum_{k=n}^{\infty}C_1 B^{k+1}\log\left|\mathcal{A}\right|\frac{2^{n-(k+1)}}{1-\frac{1}{2}}\\
      =&2^{n+1}C_1 \log\left|\mathcal{A}\right|\sum_{k=n}^\infty \left(\frac{B}{2}\right)^{k+1}\\
      =&2^{n+1}C_1 \log\left|\mathcal{A}\right|\frac{\left(\frac{B}{2}\right)^{n+1}}{1-\frac{B}{2}}\\
      =&O(B^n)
  \end{align*}
  and
  \begin{align*}
\sum_{i=1,1<i\in\mathcal{I}_{p_1}}^{n+1}\log \left|\mathcal{P}\left(\Delta^{(i,1)}_{n+1-i}, \Omega_1\times \Omega_2\right)\right|\leq&\sum_{i=1}^{n+1}C_1 B^{n+2-i}\log \left|\mathcal{A}\right|=O(B^n),
  \end{align*}
  we have 
  \[\log\left|\mathcal{P}\left(\Delta_n,X_{\Omega_1}^{p_1} \times X_{\Omega_2}^{p_2} \right)\right|-|\Delta_n|h(X_{\Omega_1}^{p_1} \times X_{\Omega_2}^{p_2})=O(B^n).\]
        \item[\bf (4)] The proof is similar to the proof of the third assertion. 
\end{proof}

\section{Entropy of axial products on $\mathbb{N}^d$ and the $d$-tree}
For $d\geq 1$, $d\geq k \geq 0$ and $p_1,...,p_k\geq 2$, let $\Omega_1,...,\Omega_k,X_{k+1},...,X_{d}$ be subshifts. Then we have the following theorem.
\begin{theorem}\label{thm5.1}The entropy of the axial product of $X_{\Omega_1}^{p_1}$, ..., $X_{\Omega_k}^{p_k}$, $X_{k+1}$, ..., $X_d$ on $\mathbb{N}^d$ is 
\begin{equation*}
    h\left(\otimes_{i=1}^k X_{\Omega_i}^{p_i} \otimes_{j=k+1}^d X_j \right)=\prod_{i=1}^k \frac{(p_i-1)^2}{p_i^2}\sum_{i_1,...,i_k=1}^\infty \frac{\lambda_{i_1,...,i_k}}{\prod_{j=1}^k p_j^{i_j-1}},
\end{equation*}
where 
\begin{equation*}
    \lambda_{i_1,...,i_k}=\lim_{i_{k+1},...,i_d \to\infty}\frac{\log\left|\mathcal{P}\left(\left\llbracket (1,...,1),\left(i_1,...,i_d\right)\right\rrbracket, \otimes_{i=1}^k \Omega_i \otimes_{j=k+1}^d X_j\right)\right|}{\prod_{j=k+1}^d i_j},
\end{equation*}
and $\mathcal{P}(\cdot,\cdot)$ is defined as in (\ref{P}).
\end{theorem}

In order to obtain the entropy of axial product of subshifts and multiplicative subshifts on the $d$-tree, we define a type for each vertex in $T^d$. For each $g\in T^d$, if $g=\cdots f_m^{i-1}$ for some $1\leq m \leq d$, then $g$ has a type $(1,...,i,...,1)$ which is a $d$-dimensional vector with the $m$-th component being $i$ and the others being 1. For $k\geq 1$, we similarly define $\Delta^{(1,...,i,...,1)}_k$. If 
\begin{equation}\label{eq--3}
    \lim_{n\to\infty}\frac{\left|\Delta_n^{(1,...,1)}\right|}{1+d+\cdots+d^n}=0,
\end{equation}
then we have the following theorem. 
\begin{theorem}\label{thm5.2}
    If (\ref{eq--3}) holds, then the entropy $h\left(\times_{i=1}^k X_{\Omega_i}^{p_i} \times_{j=k+1}^d X_j \right)$ of the axial product of $X_{\Omega_1}^{p_1}$, ..., $X_{\Omega_k}^{p_k}$, $X_{k+1}$, ..., $X_d$ on the $d$-tree is 
\begin{equation*}
    \sum_{j=1}^k\sum_{\ell=0}^{\infty} \sum_{i+\ell =\ell+1,1<i\in\mathcal{I}_{p_j}}^\infty\frac{(d-1)^2}{d^{i+\ell+1}}\log\left|\mathcal{P}\left(\Delta^{(1,...,i,...,1)}_\ell, \times_{n=1}^k \Omega_n \times_{m=k+1}^d X_m\right)\right|,
\end{equation*}
where $\mathcal{I}_{p_j}=\left\{i\in \mathbb{N}: p_j\nmid i\right\}$, and $\mathcal{P}(\cdot,\cdot)$ is defined as in (\ref{P}).

More generally, if (\ref{eq--3}) holds, then the entropy $h\left(\times_{i=1}^k X_{\Omega_i}^{S_i} \times_{j=k+1}^d X_j \right)$ is
\begin{equation*}
    \sum_{j=1}^k\sum_{\ell=0}^{\infty} \sum_{i+\ell =\ell+1,1<i\in\mathcal{I}_{S_j}}^\infty\frac{(d-1)^2}{d^{i+\ell+1}}\log\left|\mathcal{P}\left(\Delta^{(1,...,i,...,1)}_\ell, \times_{n=1}^k \Omega_n \times_{m=k+1}^d X_m\right)\right|,
\end{equation*}
where $S_i$ is a semigroup generated by $p_{i,1},...,p_{i,i_j}\in \mathbb{N}$, and $\mathcal{I}_{S_i}=\{k\in\mathbb{N}: p_{i,\ell}\nmid k,~\forall 1\leq \ell \leq i_j\}$, and $\mathcal{P}(\cdot,\cdot)$ is defined as in (\ref{P}). 
\end{theorem}
Since the proofs of Theorems \ref{thm5.1} and \ref{thm5.2} are similar to those of Theorems \ref{thm 3-1} and \ref{thm 3-2}, we omit them here.

\bibliographystyle{amsplain}
\bibliography{ban}
\end{document}